\documentclass[11pt]{amsart}
  		
\usepackage[utf8]{inputenc}
\usepackage{amsmath}
\usepackage{amsfonts}
\usepackage{amsfonts,latexsym,rawfonts,amsmath,amssymb,amsthm}
\usepackage[all,pdf]{xy}
\usepackage{amsmath,amssymb,amsthm,amscd}
\usepackage[plainpages=false]{hyperref}

\usepackage{graphicx}
\makeatletter
\renewcommand\@biblabel[1]{}
\makeatother

\numberwithin{equation}{section}

\newcommand{\beq}{\begin{equation}}
\newcommand{\eeq}{\end{equation}}
\newcommand{\beqs}{\begin{eqnarray*}}
\newcommand{\eeqs}{\end{eqnarray*}}
\newcommand{\beqn}{\begin{eqnarray}}
\newcommand{\eeqn}{\end{eqnarray}}
\newcommand{\beqa}{\begin{array}}
\newcommand{\eeqa}{\end{array}}

\DeclareMathOperator{\Ric}{Ric}

\def\lra{\longrightarrow}

\def\bc{\begin{center}}
\def\ec{\end{center}}

\def\RR{{\mathbb R}}

\def\begeq{\begin{equation}}
\def\endeq{\end{equation}}
\def\and{\quad{\rm and}\quad}

\let\lra=\longrightarrow

\def\mapright\#1{\,\smash{\mathop{\lra}\limits^{\#1}}\,}

\newtheorem{prop}{Proposition}[section]
\newtheorem{theo}[prop]{Theorem}
\newtheorem{lem}[prop]{Lemma}

\newtheorem{cor}[prop]{Corollary}

\newtheorem{defi}[prop]{Definition}

\begin{document}

\date{}

\author {Yuxing $\text{DENG}^*$ }

\thanks {* Supported by National Key R$\&$D Program of China 2022YFA1007600.}
\subjclass[2020]{Primary: 53E20; Secondary: 57R18,
 58J05}
\keywords { Ricci flow, Ricci soliton, orbifold}

\address{ Yuxing Deng\\School of Mathematics and Statistics, Beijing Institute of Technology,
Beijing, 100081, China\\
6120180026@bit.edu.cn}



\title[ Steady gradient Ricci solitons on orbifolds]{ Rigidity of positively curved Steady gradient Ricci solitons on orbifolds}

\section*{\ }

\begin{abstract} 
In this paper, we study gradient Ricci soitons on smooth orbifolds. We prove that the scalar curvature of a complete shrinking or steady gradient Ricci soliton on an orbifold is nonnegative. We also show that a complete $\kappa$-noncollapsed steady gradient Ricci soliton on a Riemannian orbifold with positive curvature operator, compact singularity and linear curvature decay must be a finite quotient of the Bryant soliton. Finally, we show that a complete steady gradient Ricci soliton on a Riemannian orbifold with positive sectional curvature must be a finite quotient of the Bryant soliton if it is asymptotically quotient cylindrical. 
\end{abstract}

\maketitle

\section{Introduction}
Steady gradient Ricci solitons are important singularity models of Type II singular solutions of the Ricci flow \cite{H2,Ham95}. A singularity model of a compact Ricci flow is always a $\kappa$-solution.  A well-known conjecture of Perelman says, 3-dimensional noncompact $\kappa$-solution must be a Ricci flow generated by the 3-dimensional Bryant soliton, which is a steady Ricci soliton (see \cite{Pe1}). This conjecture was solved by Simon Brendle (see \cite{Br1,Br3}), which is also the last missing piece of  the classification of 3-dimensional singularity models\cite{Br4}. Based on the rotationally symmetry of 3-dimensional $\kappa$-noncollapsed steady gradient Ricci solitons, Brendle proved the rotationally symmetry of 3-dimensional compact $\kappa$-solutions\cite{Br5}. Bamler and Kleiner also proved the symmetry independently\cite{BK}. Later on, Brendle, Daskalopoulos and Sesum even classfied the 3-dimensional compact $\kappa$-solutions\cite{BDS}. There are also many generalizations of these works in higher dimensions, see \cite{BDNS,BN} and references therein. 

In higher dimensions, the blow up limits of the Ricci flow can be ancient flows on orbifolds.  Appleton's work \cite{Appleton2}
shows  that flat $\mathbb{R}^4/\mathbb{Z}_2$ or
the $\mathbb{Z}_2$-quotient of the Bryant soliton may arise as the singularity models in dimension $4$. These two examples are both gradient Ricci solitons on orbifolds.  By Bamler's recent works \cite{Bam1,Bam2,Bam3}, one may expect singularity models of dimension $4$ should be ancient solutions of the Ricci flow on orbifolds with isolated singularities. So, it is interesting to study Ricci solitons on orbifolds in higher dimensions.

Recall orbifolds and Riemannian orbifolds are built on local models, see definition \ref{def-local model}, definition \ref{orbifold} and definition \ref{def-Riem orbifold}. One can also define smooth functions on orbifolds, see Definition \ref{smooth functions on orbifolds}. Then, we can define the gradient Ricci solitons on orbifold by the local model.

\begin{defi}\label{def-solitons on orbifolds}
Suppose $\mathcal{M}$ is a Riemannian orbifold equipped with metric $g$  and $f$ is a smooth function on $M$. Let $(U_{\alpha}, \hat{U}_{\alpha}, \pi_{\alpha}, \Gamma_{\alpha},\hat{g}_{\alpha})$ be an orbifold chart and there exists a smooth function $\hat{f}_{\alpha}$ defined on $\hat{U}_{\alpha}$ such that $\hat{f}_{\alpha} = f \circ \pi_{\alpha}$ by definition \ref{smooth functions on orbifolds}. Then, $(U_{\alpha},g,f)$ is called a \textbf{gradient Ricci soliton} in the chart $U_{\alpha}$ if 
\begin{align}
\hat{\rm Ric}_{\hat{g}_{\alpha}}+\frac{\lambda}{2}\hat{g}_{\alpha}={\rm Hess}_{\hat{g}_{\alpha}} \hat{f}_{\alpha}~\mbox{on}~ \hat{U}_{\alpha}.
\end{align}
$(\mathcal{M},g,f)$ is called  a \textbf{gradient Ricci soliton} if it is a gradient Ricci soliton in every chart $U_{\alpha}$.  The Ricci soliton is called shrinking, steady or expanding if $\lambda<,=,>0$ respectively.
\end{defi}

In \cite{zhang}, Zhang gave the lower bound estimate of the scalar curvature of a complete gradient Ricci soliton on smooth manifolds. In this paper, we generalize this estimate on orbifolds.

\begin{theo}\label{theo-R has lower bound}
  Suppose $(\mathcal{M},g,f)$  is a complete gradient Ricci soliton on an $n$-dimensional Riemannian orbifold.   If $\lambda\le0$, then $R(x)\ge 0$ for all $x\in M$. If $\lambda>0$, then $R\ge -C$ for some positive constant $C$. 
\end{theo}

As an application of Theorem \ref{theo-R has lower bound}, we can introduce the corresponding Ricci flow on orbifolds. Suppose $(\mathcal{M},g,f)$  is a gradient Ricci soliton on an $n$-dimensional Riemannian orbifold. We denote the topological space of $\mathcal{M}$ by $M$.  Suppose $(\hat{U},\Gamma)$ is   a local model  around $p\in U$ along with $\hat{p}\in \hat{U}$ project to $p$. Let $\pi:\hat{U}\to U$ be the projection and $\hat{f}$ be the smooth function defined in . Let $\hat{\phi}_t$ be generated by $\hat{\nabla} \hat{f}$. By Lemma \ref{lem-regualar point stay regular} and Lemma \ref{lem-singular point stay in singular set}, if $|\nabla f|(p)\neq0$, then there is a constant $\varepsilon>0$ such that 
\begin{align}
\pi(\hat{\phi}_t(\hat{p}))\in M_{\mbox{reg}},~\forall~t\in(-\varepsilon,\varepsilon),~\forall~p\in M_{\mbox{reg}},
\end{align}
and
\begin{align}
\pi(\hat{\phi}_t(\hat{p}))\in M_{\mbox{sing}},~\forall~t\in(-\varepsilon,\varepsilon),~\forall~p\in M_{\mbox{sing}}.
\end{align}
Note that $\varepsilon$ may depend on the choice of point $p$.
Then, we can define $\phi_t(p)=\pi(\hat{\phi}_t(\hat{p}))$. See  Definition \ref{def-phi-local} and Definition \ref{def-phi} for the precise definition of $\phi_t$. When $(\mathcal{M},g,f)$ is complete, $\phi_t(p)$ exists for all $t\in(-\infty,+\infty)$ and for all $p\in M$. Moreover, we have

 \begin{theo}\label{theo-automorphism}
   If $(\mathcal{M},g,f)$  is a complete gradient Ricci soliton on an orbifold, then  $\phi_t$  is an automorphism of $\mathcal{M}$ for any $t\in\mathbb{R}$.
  \end{theo}

By Theorem \ref{theo-automorphism} and Lemma \ref{lem-regualar point stay regular}, it is easy to check the following corollary.

\begin{cor}\label{cor-flow}
Let $(\mathcal{M},g,f)$  be a complete gradient Ricci soliton on an orbifold. Let $\phi_t$ be the map defined in Definition \ref{def-phi}   and $g(t)=\phi_{-t}^{\ast}g$. Then, in any orbifold chart $(U_{\alpha}, \hat{U}_{\alpha}, \pi_{\alpha}, \Gamma_{\alpha},\hat{g}_{\alpha}(t))$, the following holds
\begin{equation}
    \frac{\partial}{\partial t}\hat{g}_{\alpha}(t) = - 2 \hat{\rm Ric}_{\hat{g}_{\alpha}(t)}-\lambda \hat{g}_{\alpha}(t), ~\forall~t\in (-\infty,+\infty).
\end{equation}
Note that $(M_{\mbox{reg}},g,f)$ is an incomplete gradient Ricci soliton on manifold. On the regular locus $M_{\mbox{reg}}$, the  following holds
\begin{equation}
    \frac{\partial}{\partial t}g(t) = - 2 {\rm Ric}_{g(t)}-\lambda g(t), ~\forall~t\in (-\infty,+\infty).
\end{equation}

\end{cor}


Since we have introduced the Ricci flow on orbifolds for Ricci solitons, we may use the Ricci flow to study Ricci solitons on orbifolds. Brendle  shows that  steady Ricci solitons with positive sectional curvature  must be isometric to the Bryant soliton if it is asymptotically cylindrical \cite{Br2}.  Xiaohua Zhu and the author have also proved a rigidity result for positively curved $\kappa$-noncollpased steady Ricci solitons with linear curvature decay \cite{DZ5}. It is interesting to study whether these rults remains true on steady Ricci solions on orbifolds. Obviously, a $\mathbb{Z}_k$-quotient of the Bryant soliton satisfies our definition of gradient Ricci solitons. Such a Ricci soliton has only one singular point. This soliton also has positive sectional curvature and linear curvature decay. Moreover, it is asymptotically quotient cylindrical (See Definiton \ref{def-2}).

\begin{defi}\label{def-2}
We say that a steady gradient Ricci soliton $(M^n,g,f)$ \textbf{asymptotically (quotient) cylindrical} if for any sequence $\{p_i\}_{i\in \mathbb{N}}$ tending to infinity,
a subsequence of
$(M,R(p_i)g,p_i)$ converges to $(\mathbb{S}^{n-1}/\Gamma\times \mathbb{R},g_{\mathbb{S}^{n-1}/\Gamma}+ds^2,p_\infty)$ in the $C^\infty$ pointed Cheeger--Gromov sense, where $g_{\mathbb{S}^{n-1}/\Gamma}$ is a round metric. 
\end{defi}

In this paper, we will generalize Zhu and the author's work in \cite{DZ5} to steady gradient Ricci solitons on orbifolds that are  asymptotically quotient cylindrical. Actually, we will show that
\begin{theo}\label{main-1}
Suppose $(\mathcal{M},g,f)$  is a complete $\kappa$-noncollapsed steady gradient Ricci soliton on an $n$-dimensional Riemannian orbifold with positive curvature operator and compact singularity. If the scalar curvature has linear decay, then $(\mathcal{M},g)$ is isometric to $M_{Bry}/\Gamma$,\footnote{$M_{Bry}/\Gamma$ means a good orbifold here.} , which is a finite quotient of the $n$-dimensional Bryant soliton.
\end{theo}

We also generalize Brendle's work \cite{Br2} to Ricci solitons on orbifolds.

\begin{theo}\label{main-2}
Suppose $(\mathcal{M},g,f)$  is a complete steady gradient Ricci soliton on an $n$-dimensional Riemannian orbifold with nonnegative sectional curvature and positive Ricci curvature. If $(\mathcal{M},g,f)$ is asymptotically quotient cylindrical, then $(\mathcal{M},g)$ is isometric to $M_{Bry}/\Gamma$, i.e., a finite quotient of the $n$-dimensional Bryant soliton.
\end{theo}

When $n=4$, the assumption that $(\mathcal{M},g,f)$ is asymptotically quotient cylindrical can be replaced by assuming the tangent flow is $\mathbb{R}\times\mathbb{S}^3/\Gamma$(see \cite{BCDMZ}). Note that the symmetry group of $\mathbb{S}^{n-1}/\Gamma$ is a proper subgroup of $O(n)$ when $\Gamma$ is non-trivial. So, it is difficult to modify Brendle's work directly. A key observation is that the orbifolds in Theorem \ref{main-1} and Theorem \ref{main-2} must be good orbifolds by Theorem \ref{thm-structure of obifold solition}.

The paper is organized as follows. In Section \ref{section-basics}, we collect some basic properties of orbifolds and fix the notions. In Section \ref{section-existence of Ricci flow}, we define the map $\phi_t$ and show it is an automorphism of $\mathcal{M}$. Theorem \ref{theo-R has lower bound} and Corollary \ref{cor-flow} are also proved in Section \ref{section-existence of Ricci flow}. In Section \ref{section-rigidity of soliton}, we prove Theorem \ref{main-1} and Theorem \ref{main-2}.



\section{Basics of orbifolds}\label{section-basics}

We first review the definition of orbifolds.  Orbifolds are built on local models.
\begin{defi}\label{def-local model}
A \textbf{local model} is a pair $(\hat{U},\Gamma)$ where $\hat{U}$ a connected open subset of a Euclidean space
and $\Gamma$ is a finite group that acts smoothly and effectively on $\hat{U}$. A \textbf{soomth map} between local models $(\hat{U}_1,\Gamma_1)$ and $(\hat{U}_2,\Gamma_2)$ is given by $(\hat{f},\rho)$ so that $\hat{f}$ is $\rho$-equivalent, i.e., $\hat{f}(xg_1)=\hat{f}(x)\rho(g_1)$, where $\hat{f}:\hat{U}_{1}\to\hat{U}_2$ is a smooth map and $\rho:\Gamma_1\to \Gamma_{2}$ is a homomorphism. The smooth map is an \textbf{embedding} if $\hat{f}$ is an embedding \footnote{  When the smooth map given by $(\hat{f},\rho)$ is an embedding, $\rho$ is an injective homomorphism. See the paragraph before Definition 2.1 in \cite{KL}}.
\end{defi}

Then, orbifolds are defined as follows.
\begin{defi}\label{orbifold}
    An orbifold $\mathcal{O}$ is an maximal equivalence  class of orbifold atlases, where two atlases are equivalent if
they are both included in a third atlas. An $n$-dimensional orbifold atlas $\{(U_{\alpha},\hat{U}_{\alpha},\pi_{\alpha},\Gamma_{\alpha})\}_{\alpha\in I}$ consists of the following ,
\begin{enumerate}
    \item [(1)] A Hausdorff paracompact topological space $|\mathcal{O}|$;
     \item [(2)] An open covering $\{U_{\alpha}\}$ of $|\mathcal{O}|$;
     \item [(3)] Local models $\{(\hat{U}_{\alpha},\Gamma_{\alpha })\}$ of $|\mathcal{O}|$ with each $\hat{U}_{\alpha}$ an open subset of $\mathbb{R}^n$;
     \item [(4)] Let $\hat{\pi}:\hat{U}_{\alpha}\to \hat{U}_{\alpha}/\Gamma_{\alpha}$ be the quotient map. There exists a homomorphism $\phi_{\alpha}:\hat{U}_{\alpha}/\Gamma_{\alpha}\to U_{\alpha}$  such that $\pi_{\alpha}=\phi_{\alpha}\circ\hat{\pi}_{\alpha}$,i.e.,the following diagram commutes:
\begin{displaymath}
  \xymatrix{
  \hat{U}_{\alpha} \ar[dr]^{\pi_{\alpha}}\ar[d]_{\hat{\pi}_{\alpha}}\\
 \hat{U}_{\alpha}/\Gamma_{\alpha}\ar[r]^{\phi_{\alpha}}&U_{\alpha}
  }
  \end{displaymath}  
     \item [(5)] If $p\in  U_{\alpha}\cap U_{\beta}$, then there is a local chart $(U_{\gamma},\hat{U}_{\gamma},\pi_{\gamma},\Gamma_{\gamma})$ with $p\in U_{\gamma}\subset U_{\alpha}\cap U_{\beta}$ and embeddings \footnote{See Definition \ref{def-local model} for the definition of an embedding between local models.}$\hat{\phi}_{\gamma\alpha}:(\hat{U}_{\gamma},\Gamma_{\gamma})\to(\hat{U}_{\alpha},\Gamma_{\alpha})$ and $\hat{\phi}_{\gamma\beta}:(\hat{U}_{\gamma},\Gamma_{\gamma})\to(\hat{U}_{\beta},\Gamma_{\beta})$ such that the following diagram commutes:

     $$\begin{CD}
          \hat{U}_{\alpha}@<\hat{\phi}_{\gamma\alpha}<< \hat{U}_{\gamma}@>\hat{\phi}_{\gamma\beta}>>\hat{U}_{\beta}\\
         @VV\pi_{\alpha} V@VV\pi_{\gamma}V@VV\pi_{\beta}V\\
         U_{\alpha}@<\supset<< U_{\gamma}@>\subset>>U_{\beta}
     \end{CD}$$

     \end{enumerate}
\end{defi}

We can also define smooth functions on orbifolds.
\begin{defi}\label{smooth functions on orbifolds}
   If $\mathcal{O}^n$ is an orbifold, then a function $f:\mathcal{O}\to\RR$ is called \textbf{smooth} if for every $x \in \mathcal{O}^n$ there exists an orbifold chart $(U_{\alpha}, \hat{U}_{\alpha}, \pi_{\alpha}, \Gamma_{\alpha})$ for $\mathcal{O}^n$ with $x\in U_{\alpha}$ and there exists a smooth function $\hat{f}: \widetilde{U}_{\alpha} \to \RR$ such that $\hat{f} = f \circ \pi_{\alpha}$, i.e.the following diagram commutes:
\begin{displaymath}
  \xymatrix{
 \hat{U}_{\alpha} \ar[dr]^{\hat{f}}\ar[d]_{\pi_{\alpha}}\\
 U_{\alpha}\ar[r]^{f}&\mathbb{R}
  }
  \end{displaymath}
\end{defi}

Let $\mathcal{O}$ be an orbifold and $O$ be its topological space. Given a point $p\in O$ and a local model $(\hat{U},\Gamma)$ around $p$. Suppose $\hat{p}\in \hat{U}$ project to $p$. We can define a group $\Gamma_{p}=\{g\in \Gamma| \hat{p}g=\hat{p}\}$. The isomorphism class of $\Gamma_p$ is independent of the choice of the local model $(\hat{U},\Gamma_p)$ (See Proposition 23 in \cite{Borz}). One can also find a neighbourhood $U_p\subset O$ of $p$ and  a local model $(\hat{U}_p,\Gamma_p)$ such that $U_p=\hat{U}_p/\Gamma_p$. The regular part $O_{\mbox{reg} }$ of $O$ consists of all the points with $\Gamma_p=\{e\}$. Obviously, $O_{\mbox{reg} }$  is  a smooth manifold and is an open subset of $O$. 

Now, we  consider the singular part of $ O$. Given a local model $(\hat{U},\Gamma)$ around $p$ and assume $\hat{p}\in \hat{U}$ project to $p$. Let $H\subset \Gamma$ be a subgroup of $\Gamma$. Let $U_{H}=\{x\in \hat{U}| \Gamma_x=H\}$ and $U_{H}^{\prime}=\{x\in \hat{U}| H\subset \Gamma_x\}$. $U_{H}$ is called the \textbf{stratum} of $\hat{U}$ associated with $H$. By Proposition 32 and Remark 33 of \cite{Borz}, $U_{H}$ is a totally geodesic submanifold of $ \hat{U}$ and  $U_{H}^{\prime}$ is a closed totally geodesic submanifold of $ \hat{U}$. Hence, we can see that the singular set $O_{\mbox{sing} }$ of $O$ is locally the image of the union of a finite number of closed submanifolds of $ \hat{U}$. Therefore, $O_{\mbox{reg} }$ is an open and dense submanifold of $O$.

 Riemannian orbifolds are defined as follows.
\begin{defi}\label{def-Riem orbifold}
    A Riemannian orbifold $\mathcal{O}$ is an orbifold euqiped with a Riemannian metric $g$. $g$ is given by an atlas for $\mathcal{O}$ along with a collection of Riemannian metrics $g_{\alpha}$ on $\hat{U}_{\alpha}$ such that
    \begin{enumerate}
    \item [(1)] $\Gamma_{\alpha}$ acts isometrically on $\hat{U}_{\alpha}$;
     \item [(2)] Each $\phi_{\alpha}$ from (4) of Definition \ref{orbifold} is an isometry. 
     \item [(3)] The embeddings $\hat{\phi}_{\gamma\alpha}:(\hat{U}_{\gamma},\Gamma_{\gamma})\to(\hat{U}_{\alpha},\Gamma_{\alpha})$ and $\hat{\phi}_{\gamma\beta}:(\hat{U}_{\gamma},\Gamma_{\gamma})\to(\hat{U}_{\beta},\Gamma_{\beta})$ from (5) of Definition \ref{orbifold} are isometric.
     \end{enumerate}
\end{defi}

Now, let $\mathcal{O}$ be a Riemannian orbifold with metric $g$ and $O$ be its topological space.  To define the completeness of the Riemannian orbifold, we need to introduce the admissible curve.

\begin{defi}
Let $\mathcal{O}$ be an orbifold. For any $p\in O$ with a chart $(U_{p},\Gamma_{p})$ aroud $p$.  A curve $\gamma:[0,1]\to U_{p}$ in the chart is \textbf{admissible} if $[0,1]$ can be decomposed into countable number of subintervals $[t_i,t_{i+1}]$ such that each $\gamma|_{(t_i,t_{i+1})}$ is contained in a single  stratum $H\subset \Gamma_{p}$. A curve $\gamma:[0,1]\to O$ is \textbf{admissible} if it is admissible in every chart $U_p$ such that $\gamma\cap U_{p}\neq \emptyset$.
\end{defi}

For a admissible curve $\gamma$ on $O$, we can give a well-defined length of $\gamma$. We first decompose the curve into many small pieces such that each piece lies in a single stratum and they can be lifted in each chart. Then, we can compute the length of each piece in each chart. By adding the length of each piece, we may define the length of a adimssible curve. See Proposition 36, Proposition 37 and Theorem 38 for the details of the definition of the length in \cite{Borz}. For any $p,q\in O$, we define the distance $d(p,q)$ between $p$ and $q$ by 
\begin{align}
d(x,y)=\inf\{L(\gamma)|\gamma~\mbox{is a admissible curve from}~p~\mbox{to}~q\}.
\end{align}
Then, $(O,d)$ is a length space.
\begin{defi}
    $(\mathcal{O},g)$ is a \textbf{complete} Riemannian orbifold if $(O,d)$ is a complete length space.
\end{defi}

When $(\mathcal{O},g)$ is complete, any two points can be joined by a minimal geodesic i.e. a segment(see Theorem 9 in \cite{Borz}). The following theorem due to \cite{Borz} describes the segments on a Riemannian orbifold.

\begin{theo}[Theorem 3 and Remark 4 in \cite{Borz}]\label{theo-segment}
Let $\mathcal{O}$ be an orbifold and $O$ be its topological space. Suppose $\gamma:[0,1]\to O$ is a segment with $\gamma(0)=p$ and $\gamma(1)=q$. Then, either $\gamma\subset O_{sing}$ or $\gamma\cap O_{sing}\subset \{p\}\cup\{q\}$. 
\end{theo}

As a corrollary of the theorem, we conclude that $O_{reg}$ is convex and therefore is connected.

In this paper, we will deal with positively curved Ricci solitons. Orbifolds with positive curvature are defined as following.

\begin{defi}
    A Riemannian orbifold $\mathcal{O}$ equipped with metric $g$ has positive (or nonnegative) sectional curvature if for any $p\in O$, there is a chart $(U_{\alpha},\hat{U}_{\alpha},\pi_{\alpha},\Gamma_{\alpha},g_{\alpha})$ with $p\in U_{\alpha}$ such that $(\hat{U}_{\alpha},g_{\alpha})$ has positive ( or nonnegative) sectional curvature.
\end{defi}
    The positivity of Ricci curvature and scalar curvature can be  defined similarly.

\section{The existence of corresponding Ricci flow}\label{section-existence of Ricci flow}

We have given the definition of gradient Ricci Solitons on a Riemannian orbifold in Definition \ref{def-solitons on orbifolds}. Now, let's fix some notations that will be used later on. Let $(\mathcal{M},g,f)$ be a gradient Ricci soliton on a Riemannian orbifold. Let $M=|\mathcal{M}|$ be the topological space of $\mathcal{M}$. We denote the singular locus of $\mathcal{M}$ by $M_{sing}$ and the regular locus by $M_{reg}$. Note that $M_{reg}$ is a smooth manifold. By abuse of notation, we denote the induced metric on $M_{reg}$ by $g$ and the induced potential function on $M_{reg}$ by $f$. Then, $(M_{reg},g,f)$ is a incomplete gradient Ricci soliton on a smooth manifold.

\begin{lem}\label{lem-regualar point stay regular}
Suppose $(\mathcal{M},g,f)$  is a gradient Ricci soliton on an $n$-dimens-ional Riemannian orbifold. Suppose $(U,\hat{U},\pi,\Gamma)$ is   a chart  around $p\in U$ along with $\hat{p}\in \hat{U}$ project to $p$. Let $\hat{f}$ be the smooth function such that $\hat{f}=f\circ \pi$ on $\hat{U}$. Let $\hat{\phi}_t$ be generated by $\hat{\nabla} \hat{f}$ and $\phi_t$ be generated by $\nabla f$\footnote{We regard $\nabla f$ as the vector field on $M_{reg}$.}. If $|\nabla f|(p)\neq0$ and $p\in M_{\mbox{reg}}$, then there is a constant $\varepsilon>0$ so that 
\begin{align}
\pi(\hat{\phi}_t(\hat{p}))=\phi_t(p)\in M_{\mbox{reg}},~\forall~t\in(-\varepsilon,\varepsilon).\notag
\end{align}
\end{lem}
\begin{proof}
Since $p\in M_{\mbox{reg}}$, we can find  $U^{\prime}\subset \hat{U}$ such that $\pi|_{U^{\prime}}:U^{\prime}\to \pi(U^{\prime})$ is an isometry. Then, $\pi(U^{\prime})\subseteq M_{\mbox{reg}}$. Since $\hat{f}=f\circ\pi$ and $\pi|_{U^{\prime}}$ is an isometry, we get $\pi_{\ast}(\hat{\nabla} \hat{f})=\nabla f$. Therefore, there exists a constant $\varepsilon>0$ so that 
\begin{align}
\pi(\hat{\phi}_t(\hat{p}))=\phi_t(p)\in \pi(U^{\prime}),~\forall~t\in(-\varepsilon,\varepsilon).\notag
\end{align}
We complete the proof.
\end{proof}

\begin{lem}\label{lem-commutative lemma}
Suppose $(\mathcal{M},g,f)$  is a gradient Ricci soliton on an $n$-dimens-ional Riemannian orbifold. Suppose $(U,\hat{U},\pi,\Gamma)$ is   a chart  around $p\in U$ along with $\hat{p}\in \hat{U}$ project to $p$. Let $\hat{f}$ be the smooth function such that $\hat{f}=f\circ \pi$ on $\hat{U}$. Let $\hat{\phi}_t$ be generated by $\hat{\nabla} \hat{f}$. If  $\pi(\hat{\phi}_t(\hat{p}))\in M_{reg}$ for all $t\in (a,b)$, then, for any $h\in \Gamma$, we have  
\begin{align}
\hat{\phi}_{t}(\hat{p})\cdot h=\hat{\phi}_{t}(\hat{p}\cdot h),~\forall~t\in (a,b).
\end{align}
\end{lem}
\begin{proof}
If $\hat{\nabla} \hat{f}(\hat{p})=0$, then $\hat{\phi}_{t}(\hat{p})=\hat{p}$ for any $t\in(-\infty,+\infty)$. For any $x\in \hat{U}$, we have 
\begin{align}
   \hat{f}(x\cdot h)=f(\pi(x\cdot h))=f(\pi(x))=\hat{f}(x) .
\end{align}
Note that $h$ acts isometrically on $(\hat{U},\hat{g})$. It follows that 
\begin{align}
    |\hat{\nabla} \hat{f}|(x\cdot h)=|\hat{\nabla} \hat{f}|(x),~\forall~x\in \hat{U}.
\end{align} 
Therefore, $|\hat{\nabla} \hat{f}|(\hat{p}\cdot h)=0$ and $\hat{\phi}_{t}(\hat{p}\cdot h)=\hat{p}\cdot h$ for any $t\in(-\infty,+\infty)$.

Hence, we get 
\begin{align}
\hat{\phi}_{t}(\hat{p})\cdot h=\hat{p}\cdot h=\hat{\phi}_{t}(\hat{p}\cdot h),~\forall~t\in (-\infty,+\infty).
\end{align}

Now, we assume that $\hat{\nabla} \hat{f}(\hat{p})\neq0$. We first assume that $p\in M_{reg}$. Let $\phi_t$ be generated by $\nabla f$. Then, $|\nabla f|(p)\neq 0$.
Since $h\in \Gamma$, we have $\pi(\hat{p})=\pi(\hat{p}\cdot h)=p$.  By Lemma \ref{lem-regualar point stay regular}, we have 
\begin{align}\label{eq-1}
\pi(\hat{\phi}_t(\hat{p}))=\phi_t(p),~\forall~t\in  (a,b)
\end{align}
and 
\begin{align}\label{eq-2}
\pi(\hat{\phi}_t(\hat{p}\cdot h))=\phi_t(p),~\forall~t\in  (a,b).
\end{align}
By (\ref{eq-1}) and the definition of $\pi$, we have 
\begin{align}\label{eq-3}
\pi(\hat{\phi}_t(\hat{p})\cdot h)=\phi_t(p),~\forall~t\in (a,b).
\end{align}
By (\ref{eq-2}) and (\ref{eq-3}), we see that $\hat{\phi}_t(\hat{p}\cdot h)$ and  $\hat{\phi}_t(\hat{p})\cdot h$ are both the integral curve of $\hat{\nabla}\hat{f}$ with initial point $\hat{q}\cdot h$. By the uniqueness of the integral curve, we get  $\hat{\phi}_t(\hat{p}\cdot h)=\hat{\phi}_t(\hat{p})\cdot h$ for $t\in(-\varepsilon_1,\varepsilon_1)$. 

We complete the proof.

\end{proof}

\begin{lem}\label{lem-singular point stay in singular set}
Under the assumption of Lemma \ref{lem-regualar point stay regular}, if $|\hat{\nabla} \hat{f}|(\hat{p})\neq0$ and $p\in M_{\mbox{sing}}$, then there is a constant $\varepsilon>0$ so that 
\begin{align}
\pi(\hat{\phi}_t(\hat{p}))\in M_{\mbox{sing}},~\forall~t\in(-\varepsilon,\varepsilon).
\end{align}
\end{lem}

\begin{proof}
    We prove by contradiction. Suppose there exists a constant $\varepsilon$ so that $\pi(\hat{\phi}_{\varepsilon}(\hat{p}))\in M_{\mbox{reg}}$. We may assume that $\varepsilon>0$. Let $\hat{q}=\hat{\phi}_{\varepsilon}(\hat{p})$ and $q=\pi(\hat{q})$. Note that $q\in M_{\mbox{reg}}$. Since $|\hat{\nabla} \hat{f}|(\hat{p})\neq0$, we get that $|\hat{\nabla} \hat{f}|(\hat{\phi}_{\varepsilon}(\hat{p}))\neq0$. Therefore, $|\nabla f|(q)=|\hat{\nabla} \hat{f}|(\hat{\phi}_{\varepsilon}(\hat{p}))\neq0$.  By Lemma \ref{lem-regualar point stay regular} and the fact that $\pi(\hat{\phi}_{-\varepsilon}(\hat{q}))\in M_{\mbox{sing}}$,   there exists a constant $\varepsilon_1$ such that $0<\varepsilon_1\le\varepsilon$ and \begin{align}
        \phi_{-t}(q)\in M_{\mbox{reg}},~\forall~0\le t<\varepsilon_1.
    \end{align} 
    and 
    \begin{align}\label{eq-4}
        \pi(\hat{\phi}_{-\varepsilon_1}(\hat{q}))\in M_{\mbox{sing}}.
    \end{align} 
Since $p$ is a singular point, $\Gamma\neq \{e\}$. By (\ref{eq-4}), there exists $h\in \Gamma$ such that  $h\neq e$ and 
\begin{align}\label{eq-5}
  \hat{\phi}_{-\varepsilon_1}(\hat{q})\cdot h=\hat{\phi}_{-\varepsilon_1}(\hat{q}).  
\end{align} 
By Lemma \ref{lem-commutative lemma}, we get 
\begin{align}\label{eq-6}
\hat{\phi}_{-\varepsilon_1}(\hat{q}\cdot h)=\lim_{t\to-\varepsilon_1}\hat{\phi}_t(\hat{q}\cdot h)=\lim_{t\to-\varepsilon_1}\hat{\phi}_t(\hat{q})\cdot h=\hat{\phi}_{-\varepsilon_1}(\hat{q})\cdot h
\end{align}
By (\ref{eq-5}) and (\ref{eq-6}), we get $\hat{\phi}_{-\varepsilon_1}(\hat{q}\cdot h)=\hat{\phi}_{-\varepsilon_1}(\hat{q})$. Therefore, $\hat{q}\cdot h=\hat{q}$. It contradicts the fact that $q$ is a smooth point.
    
\end{proof}

An immediate corollary of the lemma is the following result.
\begin{cor}
Suppose $(\mathcal{M},g,f)$  is a gradient Ricci soliton on an $n$-dimensional Riemannian orbifold. Suppose $(U,\hat{U},\pi,\Gamma)$ is   a chart  around $p\in U$ along with $\hat{p}\in \hat{U}$ project to $p$. Let $\hat{f}$ be the smooth function such that $\hat{f}=f\circ \pi$ on $\hat{U}$. Let $\hat{\phi}_t$ be generated by $\hat{\nabla} \hat{f}$.  If $p$ is an isolated singular point, then $|\hat{\nabla}\hat{f}|(\hat{p})=0$.
\end{cor}
\begin{proof}
If $|\hat{\nabla}\hat{f}|(\hat{p})\neq0$, then we can find $\varepsilon>0$ such that 
\begin{align}
\pi(\hat{\phi}_t(\hat{p}))\in M_{\mbox{sing}},~\forall~t\in(-\varepsilon,\varepsilon),\notag
\end{align}
where $\hat{\phi}_t$ is generated by $\hat{\nabla}\hat{f}$. Since $p$ is an isolated singular point, it follows that $\pi(\hat{\phi}_t(\hat{p}))=p$ for all $t\in (-\varepsilon,\varepsilon)$. Then, $\hat{\phi}_t(\hat{p})=\hat{p}\cdot h$ for all $t\in (-\varepsilon,\varepsilon)$ and some $h\in \Gamma$. Then, $|\hat{\nabla}\hat{f}|(\hat{p})=0$. It contradicts the assumption that $|\hat{\nabla}\hat{f}|(\hat{p})\neq0$. Hence, $|\hat{\nabla}\hat{f}|(\hat{p})=0$.
\end{proof}

\begin{lem}\label{lem-commutative lemma-singular}
Suppose $(\mathcal{M},g,f)$  is a gradient Ricci soliton on an $n$-dimens-ional Riemannian orbifold. Suppose $(U,\hat{U},\pi,\Gamma)$ is   a chart  around $p\in U$ along with $\hat{p}\in \hat{U}$ project to $p$. Let $\hat{f}$ be the smooth function such that $\hat{f}=f\circ \pi$ on $\hat{U}$. Let $\hat{\phi}_t$ be generated by $\hat{\nabla} \hat{f}$. If  $\pi(\hat{\phi}_t(\hat{p}))\in M_{sing}$ for all $t\in (a,b)$, then, for any $h\in \Gamma$, we have  
\begin{align}
\hat{\phi}_{t}(\hat{p})\cdot h=\hat{\phi}_{t}(\hat{p}\cdot h),~\forall~t\in (a,b).
\end{align}
\end{lem}
\begin{proof}
If $\hat{\nabla} \hat{f}(\hat{p})=0$, then the proof is the same as the proof of Lemma \ref{lem-commutative lemma}.

Now, we may assume that $p\in M_{sing}$ and $\hat{\nabla} \hat{f}(\hat{p})\neq0$. It suffices to show that for any $t_0\in (a,b)$, there exists a constant $\varepsilon_0$ such that
\begin{align}
\hat{\phi}_{t}(\hat{p})\cdot h=\hat{\phi}_{t}(\hat{p}\cdot h),~\forall~t\in (t_0-\varepsilon_0,t_0+\varepsilon_0).\notag
\end{align}
Choose $\hat{p}_i$ such that   $\pi({\hat{p}_i})\in M_{reg}$ and $\hat{p}_i\to \hat{p}$ as $i\to\infty$. When $i$ large, $\hat{\phi}_t(\hat{p}_i)$ exists for $t\in (-\varepsilon_0,\varepsilon_0)$ for some $\varepsilon_0>0$. Since $\pi({\hat{p}_i})\in M_{reg}$, it follows that $\pi(\hat{\phi}_t({\hat{p}_i}))\in M_{reg}$ for $t\in (-\varepsilon_0,\varepsilon_0)$ by Lemma \ref{lem-regualar point stay regular} and Lemma \ref{lem-singular point stay in singular set}. Then,  we have $\hat{\phi}_t(\hat{p}_i\cdot h)=\hat{\phi}_t(\hat{p}_i)\cdot h$ for $t\in(-\varepsilon_0,\varepsilon_0)$ by Lemma \ref{lem-commutative lemma}. By taking $i\to\infty$, we get $\hat{\phi}_t(\hat{p}\cdot h)=\hat{\phi}_t(\hat{p})\cdot h$ for $t\in(-\varepsilon_0,\varepsilon_0)$.  We complete the proof.

\end{proof}

Let  $\phi_t$ be generated by $\nabla f$. If $p\in M_{\mbox{reg}}$, then $\phi_t(p)\in M_{\mbox{reg}}$ as long as $\phi_t(p)$ exists by Lemma \ref{lem-regualar point stay regular} and Lemma \ref{lem-singular point stay in singular set}. If $\phi_t(p)$ exists for all $t\in(-\infty,+\infty)$ and for all $p\in M_{reg}$, then $\phi_t$ is a diffeomorphism from $M_{reg}$ to $M_{reg}$. We hope to extend this diffeomorphism to be a diffeomorphism from $\mathcal{M}$ to $\mathcal{M}$.
So, we define $\phi_t(p)$ for $p\in M_{\mbox{sing}}$ in the following. 

  \begin{defi}\label{def-phi-local}
  Let $(\mathcal{M},g,f)$  be a gradient Ricci soliton on an $n$-dimensional Riemannian orbifold. For any $p\in M$, suppose $(U,\hat{U},\pi,\Gamma)$ is   a chart  around $p\in U$ along with $\hat{p}\in \hat{U}$ project to $p$. Let $\hat{f}$ be the smooth function such that $\hat{f}=f\circ \pi$ on $\hat{U}$. Let $\hat{\phi}_t$ be generated by $\hat{\nabla} \hat{f}$.  Then,  we define $\phi_t(p)=\pi(\hat{\phi}_t(\hat{p}))$ in chart $(U,\hat{U},\pi,\Gamma)$. 
  \end{defi}

  \begin{prop}\label{prop-independent of charts}
    $\phi_t(p)$ in Definition \ref{def-phi-local} is independent of the choice of charts.
  \end{prop}

  \begin{proof}
  Suppose $(U_{\alpha},\hat{U}_{\alpha},\pi_{\alpha},\Gamma_{\alpha})$ and  $(U_{\beta},\hat{U}_{\beta},\pi_{\beta},\Gamma_{\beta})$ are two charts around $p$. By Definition \ref{orbifold}, there is a local chart $(U_{\gamma},\hat{U}_{\gamma},\pi_{\gamma},\Gamma_{\gamma})$ with $p\in U_{\gamma}\subset U_{\alpha}\cap U_{\beta}$ and embeddings $\hat{\phi}_{\gamma\alpha}:(\hat{U}_{\gamma},\Gamma_{\gamma})\to(\hat{U}_{\alpha},\Gamma_{\alpha})$ and $\hat{\phi}_{\gamma\beta}:(\hat{U}_{\gamma},\Gamma_{\gamma})\to(\hat{U}_{\beta},\Gamma_{\beta})$ such that the following diagram commutes:

     $$\begin{CD}
          \hat{U}_{\alpha}@<\hat{\varphi}_{\gamma\alpha}<< \hat{U}_{\gamma}@>\hat{\varphi}_{\gamma\beta}>>\hat{U}_{\beta}\\
         @VV\pi_{\alpha} V@VV\pi_{\gamma}V@VV\pi_{\beta}V\\
         U_{\alpha}@<\supset<< U_{\gamma}@>\subset>>U_{\beta}
     \end{CD}$$
  
   Then, we get $\pi_{\gamma}=\pi_{\alpha}\circ \hat{\varphi}_{\gamma\alpha}$ on $\hat{U}_{\gamma}$. It follows that 
   \begin{align}
\hat{f}_{\gamma}=f\circ\pi_{\gamma}=f\circ(\pi_{\alpha}\circ \hat{\varphi}_{\gamma\alpha})=(f\circ\pi_{\alpha})\circ \hat{\varphi}_{\gamma\alpha}=\hat{f}_{\alpha}\circ\hat{\varphi}_{\gamma\alpha},~\mbox{on}~\hat{U}_{\gamma}. 
\end{align}
Note that $\hat{\varphi}_{\gamma\alpha}$ is a Riemannian embedding. Then, we have
\begin{align}
    (\hat{\varphi}_{\gamma\alpha})_{\ast}(\hat{\nabla} \hat{f}_{\gamma})=\hat{\nabla} \hat{f}_{\alpha}.
\end{align}
Therefore,
\begin{align}
    \hat{\varphi}_{\gamma\alpha}\Big[(\phi_{\gamma})_t(\hat{p}_{\gamma})\Big]=(\phi_{\alpha})_t(\hat{p}_{\alpha})
\end{align}
It follows that 

\begin{align}\label{eq-alpha=gamma}
    \pi_{\alpha}((\phi_{\alpha})_t(\hat{p}_{\alpha}))=(\pi_{\alpha}\circ \hat{\varphi}_{\gamma\alpha})\Big[(\phi_{\gamma})_t(\hat{p}_{\gamma})\Big]=\pi_{\gamma}((\phi_{\gamma})_t(\hat{p}_{\gamma}))
\end{align}
Similarly, we have 
\begin{align}\label{eq-beta=gamma}
    \pi_{\beta}((\phi_{\beta})_t(\hat{p}_{\beta}))=\pi_{\gamma}((\phi_{\gamma})_t(\hat{p}_{\gamma})).
\end{align}
Combing (\ref{eq-alpha=gamma}) with (\ref{eq-beta=gamma}), we get 
\begin{align}
    \pi_{\alpha}((\phi_{\alpha})_t(\hat{p}_{\alpha}))=\pi_{\beta}((\phi_{\beta})_t(\hat{p}_{\beta})).
\end{align}
We complete the proof.
  \end{proof}

\begin{prop}\label{prop-Gamma preserved}
    Let $\phi_t(p)$ be the curve defined in Definition \ref{def-phi-local}. Then, there exists a constant $\varepsilon>0$ such that\footnote{See Page 5 for the definition of $\Gamma_{p}$.}
    \begin{align}\label{monotonicity of Gamma}
\Gamma_{\phi_t(p)}=\Gamma_{p},~\forall~t\in ~(a,b),
    \end{align}
    as long as $\phi_t(p)\in U$ for $t\in(a,b)$.
    As a result, in the fundermental chart  $(U,\hat{U},\pi,\Gamma_p)$ of $p$, we have
    \begin{align}
\pi^{-1}(\phi_t(p))=\hat{\phi}_t(\hat{p}),~\forall~t\in(a,b),
    \end{align}
    as long as  $\phi_t(p)\in U$ for $t\in(a,b)$.
\end{prop}
\begin{proof}
If $|\hat{\nabla}\hat{f}|(\hat{p})=0$, then the result is trivial. So, we may suppose $|\hat{\nabla}\hat{f}|(\hat{p})\neq0$. We first claim that the following holds 
    \begin{align}
 \Gamma_{\phi_t(p)} \subset \Gamma_{p},~\forall~t\in ~(-\varepsilon,\varepsilon),
\end{align}
as long as $\phi_t(p)\in U$ for $t\in(-\varepsilon,\varepsilon)$.

Suppose not. Then,  there exist $t_0\in (-\varepsilon,\varepsilon)$  and $h\notin \Gamma_p$ such that 
\begin{align}
\hat{\phi}_{t_0}(\hat{p})\cdot h=\hat{\phi}_{t_0}(\hat{p}).
\end{align}
By Lemma \ref{lem-commutative lemma} and Lemma \ref{lem-commutative lemma-singular}, we have 
\begin{align}
\hat{\phi}_{t_0}(\hat{p})\cdot h=\hat{\phi}_{t_0}(\hat{p}\cdot h)
\end{align}
It follows that 
\begin{align}
   \hat{\phi}_{t_0}(\hat{p})= \hat{\phi}_{t_0}(\hat{p}\cdot h)
\end{align}
Hence, we get $\hat{p}=\hat{p}\cdot h$. It contradicts the fact that $h\notin \Gamma_p$.

Now, let $q=\phi_t(p)$ and $\tau=-t$. By the claim, we have 
 \begin{align}
\Gamma_p=\Gamma_{\phi_{\tau}(q)} \subset \Gamma_{q}=\Gamma_{\phi_t(p)},~\forall~t\in (-\varepsilon,\varepsilon).
\end{align}
Hence, we get 
\begin{align}
\Gamma_p=\Gamma_{\phi_t(p)},~\forall~t\in ~(-\varepsilon,\varepsilon),
\end{align}
as long as $\phi_t(p)\in U$ for $t\in(-\varepsilon,\varepsilon)$. Since we can choose $\varepsilon$ arbitrarily, it is easy to see that 
\begin{align}
\Gamma_p=\Gamma_{\phi_t(p)},~\forall~t\in (a,b),
\end{align}

Now, we suppose $(U,\hat{U},\pi,\Gamma_p)$ is a fundamental chart of $p$ and $\phi_t(p)\in U$ for all $t\in(a,b)$. For any fixed $t_0\in (a,b)$, we have 
\begin{align}
\Gamma_{\phi_{t_0}(p)}=\Gamma_p.
\end{align}
Therefore, for any $h\in \Gamma_p$, we have 
\begin{align}
    \hat{\phi}_t(\hat{p})\cdot h=\hat{\phi}_t(\hat{p}\cdot h)=\hat{\phi}_t(\hat{p}).
\end{align}

\end{proof}

In Definition \ref{def-phi-local}, we define $\phi_t(p)$ in one chart. Now, we define $\phi_t(p)$ on $M$. 

\begin{defi}\label{def-phi}
  Let $(\mathcal{M},g,f)$  be a gradient Ricci soliton on an $n$-dimensional Riemannian orbifold. For any $p\in M$, let $\phi_t(p)=\phi_{t_m}(\phi_{t_{m-1}}(\cdots\phi_{t_1}(p)\cdots))$ if the following hold:
  
  (1) $t=t_1+t_2+\cdots+t_m$ and $t\cdot t_{i}\ge 0$ for all $1\le i\le m$;

  (2) For $1\le i\le m$, there exists a sequence of fundamental charts $(U_i,\hat{U}_i,\pi_i,\Gamma_i)$ such that  $p_1=p$, $p_{i+1}=\phi_{t_{i}}(p_i)$  and $\phi_{s}(p_i)\in U_i$  when $s\in[0,t_i]$, where $\phi_{s}(p_i)\in U_i$ is defined in Definition \ref{def-phi-local} .
\end{defi}

We need to show Definition \ref{def-phi} is well-defined.

\begin{theo}
 Definition \ref{def-phi} is well-defined.
\end{theo}
\begin{proof}
It suffices to deal with the case that $t\cdot t_i>0$ for all $1\le i\le m$ in Definition \ref{def-phi}.  Without loss of generality, we may assume $t>0$. Supose the following hold:

(1) $t=t_1+t_2+\cdots+t_m$ and $t\cdot t_{i}> 0$ for all $1\le i\le m$;

  (2) there exists a chart $(U_i,\hat{U}_i,\hat{\pi}_i,\Gamma_i)$ such that $p_i,p_{i+1}\in U_i$ for every $1\le i\le m$, where $p_1=p$ and $p_{i+1}=\hat{\pi}_{i}(\hat{\Phi}^i_{t_{i}}( \hat{p}_i))$ with $\hat{\pi}_{i}( \hat{p}_i)=p_i$ for every $1\le i\le m$. Each $\hat{\Phi}^i_{s}( \hat{p}_i)$ is generated by $\hat{\nabla} (f\circ \hat{\pi}_i)$ here.
  
  (3) $t=s_1+s_2+\cdots+s_l$ and $t\cdot s_{j}> 0$ for all $1\le j\le l$;

  (4) there exists a chart $(V_j,\tilde{V}_j,\tilde{\pi}_j,\Gamma_j)$ such that $q_j,q_{j+1}\in V_j$ for every $1\le j\le l$, where $q_1=p$ and $q_{j+1}=\tilde{\pi}_{j}(\tilde{\Psi}^j_{t_{j}}( \tilde{q}_j))$ with $\tilde{\pi}_{j}( \tilde{q}_j)=p_j$ for every $1\le j\le l$. Each $\tilde{\Psi}^j_{s}( \tilde{q}_j)$ is generated by $\tilde{\nabla} (f\circ \tilde{\pi}_j)$ here. 

  Now, we need to show that $p_{m+1}=q_{l+1}$.

  Let $T_i=t_1+t_2+\cdots+t_i$ for $1\le i\le m$ and $S_j=s_1+s_2+\cdots+s_j$ for $1\le j\le l$. Set $T_0=S_0=0$. Then, $\{[T_i,T_{i+1}]\}_{0\le i\le m-1}$  and $\{[S_j,S_{j+1}]\}_{0\le j\le l-1 }$ are both decompositions of  $[0,t]$. Let $\{[D_k,D_{k+1}]\}_{0\le k\le N}$ be a refinement decomposition of $[0,t]$ such that $D_0=0$, $D_N=t$, $D_k<D_{k+1}$ for $0\le k\le N-1$ and $\{D_0,D_1,\cdots,D_N\}=\{S_0,S_1,\cdots,S_l\}\cup\{T_0,T_1,\cdots,T_m\}$. Let $d_k=D_k-D_{k-1}$ for $1\le k\le N$. 
  
  By the choice of $D_k$, there exist constants $k_1$, $k_2$, $\cdots$, $k_m$ such that $D_{k_i}=T_i$ for $1\le i\le m$. Then, $k_m=N$. Now, we let $k_0=0$. For $1\le i\le m$ and ${k_{i-1}}+1\le k\le k_{i}$, we define 
\begin{equation}\begin{cases}
      x_{k_{i-1}+1}=p_{i},&~\mbox{when} ~k=k_{i-1}+1,\\
      x_{k}=\Phi^i_{d_{k-1}}(x_{k-1}),&~\mbox{when} ~k_{i-1}+2\le k\le k_{i},
  \end{cases}
  \end{equation}
  where  $\hat{\Phi}^i_s(\hat{p}_i)$ is the integral curve of $\hat{\nabla} (f\circ \hat{\pi}_i)$ in $\hat{U}_i$ and $\Phi^i_s(x)=\hat{\pi}_i(\hat{\Phi}^i_s(\hat{x}))$ if $x=\hat{\pi}_i(\hat{x})$. Now, we fix $i$. Note that 
 \begin{align}
d_{k_{i-1}+1}+d_{k_{i-1}+2}+\cdots+d_{k_i}=D_{k_i}-D_{k_{i-1}}=T_i-T_{i-1}=t_i.
 \end{align}
  Since $\hat{\Phi}^i_s(\hat{p}_i)$ is the integral curve of $\hat{\nabla} (f\circ \hat{\pi}_i)$ in $\hat{U}_i$ when  $s\in[0,t_i]$, we have 
 \begin{align}
\hat{\Phi}^i_{t_i}(\hat{p}_i)=\hat{\Phi}^i_{d_{k_{i-1}+1}+d_{k_{i-1}+2}+\cdots+d_{k_i}}(\hat{p}_i)=\hat{\Phi}^i_{d_{k_i}}(\hat{\Phi}^i_{d_{k_i-1}}(\cdots\hat{\Phi}^i_{d_{k_{i-1}+1}}(\hat{p}_i)\cdots))
 \end{align}
 It follows that 
 \begin{align}
    x_{k_i+1}=p_{i+1}={\Phi}^i_{t_i}({p}_i)={\Phi}^i_{d_{k_i}}({\Phi}^i_{d_{k_i-1}}(\cdots{\Phi}^i_{d_{k_{i-1}+1}}({p}_i)\cdots))
 \end{align}
 Then, we have 
 \begin{align}
 x_{k_i+1}={\Phi}^i_{d_{k_i}}(x_{k_i}),~\forall~1\le i\le m.
 \end{align}
Note that $k_m=N$. Then, we get
\begin{align}
x_{k+1}=\Phi_{d_{k}}(x_k),~\forall~1\le k\le N.
  \end{align}
 where $\Phi_{d_k}(\cdot)=\Phi^i_{d_k}(\cdot)$ if $k_{i-1}+1\le k\le k_{i}$.

On the other hand, by the choice of $D_k$, there exist constants $\bar{k}_1$, $\bar{k}_2$, $\cdots$, $\bar{k}_l$ such that $D_{\bar{k}_j}=S_j$ for $1\le j\le l$. Then, $\bar{k}_l=N$. Let $\bar{k}_0=0$.  For $1\le j\le l$ and ${\bar{k}_{j-1}}+1\le k\le \bar{k}_{j}$, we define 
\begin{equation}\begin{cases}
      y_{\bar{k}_{j-1}+1}=p_{j},&~\mbox{when} ~k=\bar{k}_{j-1}+1,\\
      y_{k}=\Psi^i_{d_{k-1}}(y_{k-1}),&~\mbox{when} ~k_{j-1}+2\le k\le k_{j},
  \end{cases}
  \end{equation}
  where  $\tilde{\Psi}^i_s(\tilde{p}_i)$ is the integral curve of $\tilde{\nabla} (f\circ \tilde{\pi}_i)$ in $\tilde{V}_j$ and $\Psi^i_s(x)=\tilde{\pi}_i(\tilde{\Psi}^i_s(\tilde{x}))$ if $x=\tilde{\pi}_i(\tilde{x})$. Similarly, we  can show that
\begin{align}
y_{k+1}
=\Psi_{d_{k}}(y_k),~\forall~1\le k\le N.
  \end{align}
 where $\Psi_{d_k}(\cdot)=\Psi^j_{d_k}(\cdot)$ if $\bar{k}_{j-1}+1\le k\le \bar{k}_{j}$.

 Now, it suffices to show that 
 \begin{align}
     x_{k_m+1}=y_{\bar{k}_l+1}.
 \end{align}
 
We claim that $x_k=y_k$ for all $1\le k\le N+1$. We prove the claim by induction on $k$. When $k=1$, we have $x_1=p=y_1$ by assumption. Suppose $x_k=y_k$. Then, there exist an integer $i$ and $j$ such that 
\begin{align}
    k_{i-1}+1\le k+1\le k_i.
\end{align}
Then, we have 
\begin{equation}
      x_{k+1}=\hat{\pi}_i(\hat{\Phi}^i_{d_k}(\hat{x}_k)),
  \end{equation}
  where  $\hat{\Phi}^i_{s}(\hat{x}_k)$ is the integral curve of $\hat{\nabla} (f\circ \hat{\pi}_i)$ with  $\hat{\Phi}^i(\hat{x}_k)=x_k$  and $\hat{\Phi}^i_{s}(\hat{x}_k)\in\hat{U}_i$ for $s\in[0, d_k]$.

  Similarly, there exist an integer $j$ such that 
  \begin{align}
    \bar{k}_{j-1}+1\le k+1\le \bar{k}_j.
\end{align}
Then, we have 
\begin{equation}    y_{k+1}=\tilde{\pi}_i(\tilde{\Psi}^i_{d_k}(\tilde{y}_k)),
  \end{equation}
  where  $\tilde{\Psi}^i_{s}(\tilde{y}_k)$ is the integral curve of $\tilde{\nabla} (f\circ \tilde{\pi}_i)$ with $\tilde{\Psi}^i(\tilde{y}_k)=y_k$  and $\tilde{\Psi}^i_{s}(\tilde{y}_k)\in\hat{V}_j$ for $s\in[0, d_k]$.

    By Proposition \ref{prop-independent of charts}, we have $x_{k+1}=y_{k+1}$. 
We complete the proof.
    
\end{proof}

  \begin{lem}\label{lem-admissable}
       $\phi_t(p)$, $t\in [a,b]$ is an admissible curve.
  \end{lem}

  \begin{proof}
  By Definition \ref{def-phi},    the following hold:
  
  (1) $b-a=t_1+t_2+\cdots+t_m$ and $t\cdot t_{i}\ge 0$ for all $1\le i\le m$;

  (2) there exists a sequence of fundamental charts $(U_i,\hat{U}_i,\pi_i,\Gamma_i)$ such that $p_i,p_{i+1}\in U_i$ for every $1\le i\le m$, where $p_1=\phi_a(p)$, $p_{i+1}=\phi_{t_{i}}(p_i)$ for $1\le i\le m$.

  By Proposition \ref{prop-Gamma preserved}, 
  \begin{align}
      \Gamma_{\phi_s(p_i)}=\Gamma_{p_i},~\forall~s\in~[0,t_i].
  \end{align}
  Then, $\phi_t|_{(t_i,t_{i+1})}$ is contained in a single stratum associated with $\Gamma_{p_i}\subset \Gamma_i$.
  Hence, $\phi_t(p)$, $t\in [a,b]$ is an admissible curve.

  \end{proof}

  To introduce the Ricci flow on $(M,g,f)$, it remains to show that $\phi_t(p)$ exists for all $p\in M$ and $t\in(-\infty,+\infty)$. We first prove the long time existence  when $R\ge -C$ for some positive constant $C$.

\begin{lem}\label{lem-completeness when R has lower bound}
   Suppose 
 $(\mathcal{M},g,f)$ is a complete Ricci soliton and $R\ge -C$ for some positive constant $C$. Let $\phi_t$ be generated by $\nabla f$. If $p\in M$, then $\phi_t(p)$ exists for all $t\in (-\infty,+\infty)$. 
\end{lem}
\begin{proof}
This lemma is due to Bennett Chow(See Theorem 2.27 in \cite{Chow}).If $|\nabla f(p)|=0$, then $\phi_t(p)=p$ for all $t\in (-\infty,+\infty)$. It remains to deal with the case that  $|\nabla f(p)|\neq0$.  By  Lemma \ref{lem-admissable}, $\phi_t(p)$ is an admissible curve. So, one can compute the length of $\phi_t(p)$ for $t\in [a,b]$ in the local chart even if $p\in M_{sing}$. Note that $f(\phi_t(p))$ and $|\nabla f|^2(\phi_t(p))$ are both smooth functions on $t$ for any fixed $p\in M$. So, the argument of of Lemma 2.28 in \cite{Chow} remains to be valid on orbifolds.

It is well known that 
\begin{align}
R(x)+|\nabla f|^2(x)+\lambda f(x)=C_1,~\forall~x\in M_{\mbox{reg}}\notag
\end{align}
where $C_1$ is a constant. Since $M_{\mbox{reg}}$ is dense in $M$, we have 
\begin{align}\label{Hamilton's identity}
R(x)+|\nabla f|^2(x)+\lambda f(x)=C_1,~\forall~x\in M.
\end{align}

By (\ref{Hamilton's identity}), following the argument of Lemma 2.28 in \cite{Chow}, one can show that 
\begin{align}
d(\phi_t(p),p)\le C|t|e^{C|t|}.\notag
\end{align}
Hence, $\phi_t(p)$ exists for all $t\in\mathbb{R}$

\end{proof}

Now, we  remove the lower bound assumption in Lemma \ref{lem-completeness when R has lower bound}.

To prove Theorem\ref{theo-R has lower bound}. We introduce the following lemma which is a modification of Proposition 2.2 in \cite{zhang}.

\begin{lem}\label{lem-bound of laplacian d}
Let $(M,g,f)$ be a gradient Ricci soliton on a connected Riemannian manifold. Fix $p\in M$, let $d(x)=\inf \{\mbox{Length of~} \gamma:\gamma\mbox{~is a smooth}$ 
$ \mbox{curve in~}M\mbox{~that connecting~}p\mbox{~and~}x\}$.Suppose $\mbox{Ric}\le (n-1)K$ on $B(p,r_0)$ for some constant $r_0$. Suppose $x_1\notin B(p,r_0)$ and there is a minimal geodesic in $M$ connecting $x_1$ and $p$. If $x_1$ is not in the cut locus of $p$, then  we have 
\begin{align}\label{estimate of delta d}
\Delta d+\langle\nabla f,\nabla d\rangle\le \frac{\lambda}{2}+(n-1)(\frac{2}{3}Kr_0+r_0^{-1})+|\nabla f|(p),~\forall~x\in B(x_1,\varepsilon),
\end{align}
where $\varepsilon$ is a positive constant.
\end{lem}
\begin{proof}
Compared with Proposition 2.2 in \cite{zhang}, $(M,g)$ may be incomplete in our case.  Note there is a minimal geodesic in $M$ connecting $x_1$ and $p$ and $x_1$ is not in the cut locus of $p$. So for $x\in B(x_1,\varepsilon)$, there exists a minimal geodesic connecting $x$ and $p$ for some $\varepsilon>0$. Then $d(x)$ can be regarded as the usual distance function on  $B(x_1,\varepsilon)$. Let $\gamma:[0,d(x)]\to M$ be a minimal geodesic from $p$ to $x$. Let 
$\{e_1,e_2,\cdots,e_{n-1},\gamma^{\prime}(0)\}$ be an orthonormal basis $\{e_1(s),e_2(s),\cdots,e_{n-1}(s),\gamma^{\prime}(s)\}$ along $\gamma$. Since $x$ is not in the cut locus of $p$, we can find Jacobian fields along $\gamma$ with $X_i(0)=0$ and $X_{i}(d(x))=e_i(d(x))$ for $i=1,2,\cdots,n-1$. Then, one can  follow the argument and the computation in Proposition 2.2 of \cite{zhang} to show (\ref{estimate of delta d}) by using index comparison theorem and the soliton equation.

\end{proof}

Now, we are ready to prove Theorem\ref{theo-R has lower bound}. The proof is similar to Theorem 1.3 in \cite{zhang}. In the proof of Theorem 1.3 in \cite{zhang}, the author constructed a function $u$ on $M$ and applying the maximum principle to $u$. In our case, the main difference occurs when $u$ attains its minimal at a singular point $x_1\in M_{\mbox{sing}}$.

\begin{proof}[Proof of Theorem\ref{theo-R has lower bound}.]

We first prove the case $\lambda\le 0$. Fix a point $p\in M_{\mbox{reg}}$. Note that there is a positive constant $r_0$ such that $\mbox{Ric}\le (n-1)r_0^{-2}$ on $B(p,r_0)$ and $|\nabla f|(p)\le (n-1)r_0^{-1}$. For any fixed constant $A>2$, we consider the function $u(x)=\psi(\frac{d(x,p)}{Ar_{0}})R(x)$, where $\psi(s)$ is a smooth nonnegative decreasing function such that $\psi(s)=1$ on $(-\infty,\frac{1}{2}]$ and $\psi(s)=0$ on $[1,\infty)$. We can also assume that there is a constant $C_1>0$ such that \begin{align}\label{bounds of psi}
    |\psi^{\prime}(s)|\le C_1,~[\psi^{\prime}(s)]^2\le C_1  \psi(s),~|\psi^{\prime\prime}(s)|\le C_1,~\forall~s\in (-\infty,+\infty).
\end{align}

If $\min_{x\in M}u\ge 0$, then $R(x)\ge 0$ on $B(p,\frac{Ar_0}{2})$.

If $\min_{x\in M}u< 0$, then there is a point $x_1\in B(p,Ar_0)$ such that $u(x_1)=\min_{x\in M}u< 0$. Since $\psi(s)\ge 0$, we get 
\begin{align}\label{theorem-lower bound-eq-5}
    R(x_1)<0.
\end{align}
If $x_1$ is a smooth point of $u$, then we also have
\begin{align}
    |\nabla u|(x_1)=0,~\Delta u(x_1)\ge0.\notag
\end{align}

\textbf{Case 1:} $x_1\in B(p,r_0)$. 

In this case, $\frac{d(x_1,p)}{Ar_0}<\frac{1}{2}$. There exist a small constant $\varepsilon>0$ such that $\frac{d(x,p)}{Ar_0}<\frac{1}{2},$ for all $x\in B(x_1,\varepsilon)$. Hence, $u(x)=R(x)$ on $B(x_1,\varepsilon)$. $R(x)$ attains its minimum at $x_1$ on $B(x_1,\varepsilon)$.

Suppose $x_1\in M_{\mbox{sing}}$. Then,  we can find    a local model $(\hat{U},\Gamma)$ around $x_1\in U$ along with $\hat{x_1}\in \hat{U}$ project to $x_1$.  Then, there is a smooth function $\hat{f}$ defined on $\hat{U}$, such that 
\begin{align}\label{local ricci soliton equation}
\hat{\rm Ric}+\frac{\lambda}{2}\hat{g}={\rm Hess}_{\hat{g}} \hat{f}~\mbox{on}~ \hat{U}.
\end{align}
Let $\pi:\hat{U}\to U$ be the projection from $\hat{U}$ to $U$ with $\pi(\hat{x_1})=x_1$. We may assume that $U\subset B(x_1,\varepsilon)$. Then, 
    \begin{align}
\hat{R}(x)=R(\pi(x)),~\forall~x\in U\subset B(x_1,\varepsilon).\notag
    \end{align}
    Hence, $\hat{R}(x)$ attains its minimum at $\hat{x}_1$ on $\hat{U}$ with 
    \begin{align}\label{theorem-lower bound-eq-4}
\hat{R}(\hat{x}_1)=R(x_1)=u(x_1)<0.
    \end{align}
    Then, 
\begin{align}\label{theorem-lower bound-eq-1}
    |\hat{\nabla} \hat{R}|(x_1)=0,~\hat{\Delta}\hat{R}(x_1)\ge0.
\end{align}
By (\ref{local ricci soliton equation}), we get 
\begin{align}\label{theorem-lower bound-eq-2}
\hat{\Delta}\hat{R}+\langle \hat{\nabla} \hat{f},\hat{\nabla}\hat{R}\rangle+\lambda \hat{R}+|\hat{\mbox{Ric}}|^2=0~on~\hat{U}.
\end{align}
Note that 
\begin{align}\label{theorem-lower bound-eq-3}
|\hat{\mbox{Ric}}|^2(x)\ge \frac{2}{n}\hat{R}^2(x),~\forall~x\in~\hat{U}.
\end{align}
Combining  (\ref{theorem-lower bound-eq-4}), (\ref{theorem-lower bound-eq-1}), (\ref{theorem-lower bound-eq-2}) and (\ref{theorem-lower bound-eq-3}), we get
\begin{align}
\hat{R}(\hat{x_1})\ge -\frac{n\lambda}{2}.\notag
\end{align}
If $R(x)\le 0$ and $x\in B(p,\frac{Ar_0}{2})$, then
\begin{align}
R(x)\ge \frac{u(x)}{\psi(\frac{1}{2})}\ge \frac{u(x_1)}{\psi(\frac{1}{2})}=\frac{\hat{R}(\hat{x}_1)}{\psi(\frac{1}{2})}\ge -\frac{n\lambda}{2\psi(\frac{1}{2})}.\notag
\end{align}
Hence,
\begin{align}\label{lower bound estimate of R-1-1}
R(x)\ge \min\{0,-\frac{n\lambda}{2\psi(\frac{1}{2})}\},~\forall~x\in B(p,\frac{Ar_0}{2}).
\end{align}

Now, we suppose $x_1\in M_{\mbox{reg}}$. Since $R(x)$ satisfies the soliton equation  on $B(x_1,\varepsilon)$, we get
\begin{align}
\Delta R+\langle \nabla f,\nabla R\rangle+\lambda R+|\mbox{Ric}|^2=0~on~M_{\mbox{reg}}.\notag
\end{align}
By a similar argument, we can also show that 
\begin{align}
R(x_1)\ge -\frac{n\lambda}{2}.\notag
\end{align}
Therefore,
\begin{align}\label{lower bound estimate of R-1-2}
R(x)\ge \min\{0,-\frac{n\lambda}{2\psi(\frac{1}{2})}\},~\forall~x\in B(p,\frac{Ar_0}{2}).
\end{align}

\textbf{Case 2:} $x_1\notin B(p,r_0)$.

(1) Suppose $x_1\in M_{sing}$.    We can find    a local model $(\hat{U},\Gamma)$ around $x_1\in U$ along with $\hat{x_1}\in \hat{U}$ project to $x_1$.  Then, there is a smooth function $\hat{f}$ defined on $\hat{U}$, such that 
\begin{align}\label{local ricci soliton equation-2}
\hat{\rm Ric}+\frac{\lambda}{2}\hat{g}={\rm Hess}_{\hat{g}} \hat{f}~\mbox{on}~ \hat{U}.
\end{align}
Let $\pi:\hat{U}\to U$ be the projection from $\hat{U}$ to $U$ with $\pi(\hat{x_1})=x_1$. Then, 
    \begin{align}
\hat{R}(x)=R(\pi(x)),~\forall~x\in U.\notag
    \end{align}

    Now, we construct a new Ricci soliton $(N,\bar{g},\bar{f})$ on smooth manifold which may be incomplete and introduce a new function $\bar{u}$ on $N$. By Theorem \ref{theo-segment}, $M_{\mbox{reg}}$ is geodesically convex. Then, we can find a minimal geodesic $\gamma(s):[0, d(x_1,p)]\to M$ from $p$ to $x_1$ and $\gamma(s)$ is a smooth point for all $s\in[0,d(x_1,p))$. Choose $s_1\in (0,d(x_1,p))$ such that $\gamma(s_1)\in U$. Let $q=\gamma(s_1)$ and $\delta=d(x_1,q)$. Let $L=\{\gamma(s)\in M:s\in[0, d(x_1,p)]\}$. Let $\varepsilon$ be a positive constant such that 
 $\varepsilon<<\delta$. Let $B(L,\varepsilon)=\{x\in M:d(x,L)\le \varepsilon\}$. Suppose $\hat{q}\in \hat{U}$ and $\pi(\hat{q})=q$. Since $q$ is a smooth point, $\pi|_{\hat{B}(\hat{q},\varepsilon)}:\hat{B}(\hat{q},\varepsilon)\to B(q,\varepsilon)$ is an isomorphism when $\varepsilon$ is small enough. Let $D=B(L,\varepsilon)\cap\partial B(x_1,\delta) $. Note that $q$ is a smooth point. When $\delta$ is small and $\varepsilon<<\delta$, we can get that $D$ is diffeomorphic to an $(n-1)$-dimensional  disk with a smooth boundary. Let $\hat{D}=(\pi|_{\hat{B}(\hat{q},\varepsilon)})^{-1}(D)$. Then, $\pi|_{\hat{D}}:(\hat{D},\hat{g})\to (D,g)$ is also an isomorphism, where $(\hat{D},\hat{g})$ and $(D,g)$ carry the induced metrics from $(\hat{U},\hat{g})$ and $(M,g)$ respectively.  Let $\bar{N}$ be the space by gluing $B(L,\varepsilon)\setminus B(x_1,\delta)$ and $\hat{B}(\hat{x_1},\delta)$ via the isomorphism $\pi|_{\hat{D}}:(\hat{D},\hat{g})\to (D,g)$. That is to say,  $\bar{N}=(B(L,\varepsilon)\setminus B(x_1,\delta))\sqcup_{\pi|_{\hat{D}}}\hat{B}(\hat{x_1},\delta)$. Then, the interior $N$ of $\bar{N}$ is a manifold. See Figure \ref{fig:enter-label} for the construction of $N$. The red region in the picture is $D$ or $\hat{D}$.

 \begin{figure}
     \includegraphics[width=8cm]{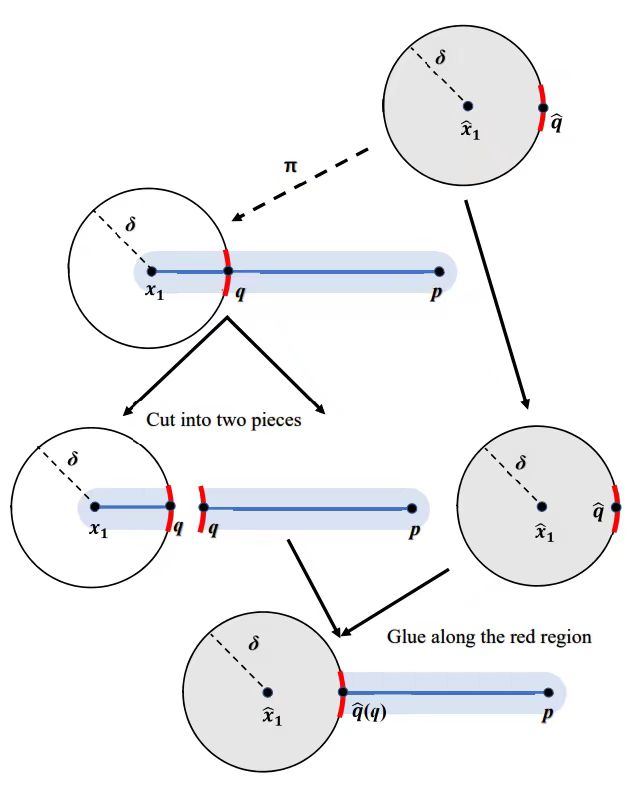}
     \caption{surgury on the soliton}
     \label{fig:enter-label}
 \end{figure}

 Let $\bar{g}(x)=\hat{g}(x)$ for $x\in N\cap \hat{B}(\hat{x_1},\delta)$ and $\bar{g}(x)=g(x)$ for $x\in N\cap(B(L,\varepsilon)\setminus B(x_1,\delta))$. Then, $(N,\bar{g})$ is a Riemannian manifold. Let $\bar{f}(x)=\hat{f}(x)$ for $x\in N\cap \hat{B}(\hat{x_1},\delta)$ and $\bar{f}(x)=f(x)$ for $x\in N\cap(B(L,\varepsilon)\setminus B(x_1,\delta))\subset M_{\mbox{reg}}$. Then, $\bar{f}$ is smooth on $N$. By (\ref{local ricci soliton equation-2}) and soliton equation on regular set,  $\bar{f}$ satisfies
\begin{align}\label{soliton equation on N}
\bar{\rm Ric}+\frac{\lambda}{2}\bar{g}={\rm Hess}_{\bar{g}} \bar{f}~\mbox{on}~ N.
\end{align}
Let $\bar{d}$ be the distance function on $(N,\bar{g})$. Note that $\gamma|_{[s_1,d(x_1,p)]}$ is a minimal geodesic from $q$ to $x_1$ on $M$ and $\gamma(s)\in M_{\mbox{reg}}$ for $s\in [s_1,d(x_1,p))$.  When $\delta $ is small, there is a geodesic $\hat{\gamma}(s):[s_1,d(x_1,p)]\to \hat{B}(\hat{x_1},\delta) $ from $\hat{q}$ to $\hat{x_1}$ such that 
\begin{align}
\pi(\hat{\gamma}(s))=\gamma(s),~\forall~s\in[s_1,d(x_1,p)].
\end{align}
Now, we let $\bar{\gamma}(s)=\gamma(s)$ for $s\in[0,s_1]$ and $\bar{\gamma}(s)=\hat{\gamma}(s)$ $s\in [s_1,d(x_1,p)]$. Then, $\bar{\gamma}(s)$ is a geodesic on $N$. Since $\gamma$ is minimal on $M$, we get $\bar{\gamma}$ is a minimal geodesic on $N$. Then, 
\begin{align}\label{equal distance at x_1}
    \bar{d}(\hat{x_1},p)=d(x_1,p).
\end{align}

\textbf{i)} Suppose $\hat{x}_1$ is not in the cut locus of $p$ on $(N,\bar{g})$. Then, $\bar{d}(x,p)$ is smooth on $\hat{B}(\hat{x_1},\delta^{\prime})$ for some positive constant $\delta^{\prime}$. Note that 
\begin{align}\label{distance comparison}
\bar{d}(x,p)\ge d(\pi(x),p),~\forall~x\in \hat{B}(\hat{x_1},\delta).
\end{align}
Recall that $\hat{R}(x_1)=R(x_1)<0$.  We may assume $\delta=\delta^{\prime}$. Then,
\begin{align}\label{scalar curvature negative}
    \hat{R}(x)<0,~\forall~x\in \hat{B}(\hat{x_1},\delta).
\end{align}
Now, let $\bar{u}(x)=\psi(\frac{\bar{d}(x,p)}{Ar_0})\bar{R}(x)$ for $x\in \hat{B}(\hat{x_1},\delta)$. Since $\psi$ is nonnegative and nonincreasing, by (\ref{distance comparison}) and (\ref{scalar curvature negative}), we have
\begin{align}
    \bar{u}(x)\ge u(\pi(x)),~\forall~x\in \hat{B}(\hat{x_1},\delta).
\end{align}
By (\ref{equal distance at x_1}), we also have 
    \begin{align}
    \bar{u}(x_1)=\psi(\frac{\bar{d}(\hat{x}_1,p)}{Ar_0})\bar{R}(\hat{x_1})=\psi(\frac{d(x_1,p)}{Ar_0})R(x_1)=u(x_1)<0.
\end{align}
Note that $u$ attains its minimum at $x_1$. Hence, $\bar{u}(x)$ attains its minimum at $\hat{x}$. Note that $\bar{u}(x)$ is smooth on $\hat{B}(\hat{x_1},\delta^{\prime})$. Then,   we  have
\begin{align}\label{bar u at minimum}
    |\bar{\nabla} \bar{u}|(\hat{x}_1)=0,~\bar{\Delta} \bar{u}(\hat{x}_1)\ge0.
\end{align}
Therefore,
\begin{align}
    \bar{\nabla}[ \psi(\frac{\bar{d}(x,p)}{Ar_0})\bar{R}(x)]|_{x=\hat{x}_1}=0.\notag
\end{align}
Hence,
\begin{align}\label{computation of laplacian u-1}
\bar{\nabla} \bar{R}(\hat{x}_1)=-\frac{\psi^{\prime}}{\psi^2}\cdot \frac{\bar{\nabla}\bar{d}(\hat{x}_1,p)}{Ar_0}\cdot \bar{u}(\hat{x}_1).
\end{align}
By (\ref{soliton equation on N}), we also have
\begin{align}\label{computation of laplacian u-2}
\bar{\Delta} \bar{R}+\langle \bar{\nabla} \bar{f},\bar{\nabla} \bar{R}\rangle+\lambda \bar{R}+|\bar{\mbox{Ric}}|^2=0~on~N.
\end{align}
By (\ref{computation of laplacian u-1}) and (\ref{computation of laplacian u-2}), we get 
\begin{align}\label{computation of laplacian u-3}
    \bar{\Delta}\bar{u}(\hat{x}_1)=&\big(\frac{\psi^{\prime\prime}}{\psi}\cdot\frac{1}{(Ar_0)^2}-\frac{(\psi^{\prime})^2}{\psi^2}\cdot\frac{2}{(Ar_0)^2}-\lambda\big)\bar{u}(\hat{x}_1)-\psi|\bar{\mbox{Ric}}|^2\\
    &+\frac{\psi^{\prime}}{\psi}\cdot\frac{\bar{u}(\hat{x_1})}{Ar_0}\big(\bar{\Delta}\bar{d}+\langle \bar{\nabla}\bar{f},\bar{\nabla}\bar{d}\rangle\big)\notag\\
    \le&\big(\frac{\psi^{\prime\prime}}{\psi}\cdot\frac{1}{(Ar_0)^2}-\frac{(\psi^{\prime})^2}{\psi^2}\cdot\frac{2}{(Ar_0)^2}-\lambda\big)\bar{u}(\hat{x}_1)-\frac{2}{n}\psi\bar{R}^2\notag\\
    &+\frac{\psi^{\prime}}{\psi}\cdot\frac{\bar{u}(\hat{x_1})}{Ar_0}\big(\bar{\Delta}\bar{d}+\langle \bar{\nabla}\bar{f},\bar{\nabla}\bar{d}\rangle\big)\notag
\end{align}
Recall that 
\begin{align}
\bar{\mbox{Ric} }(x)&\le \frac{n-1}{r_0^2},~\forall~x\in \bar{B}(p,r_0).\label{computation of laplacian u-4}\\
|\bar{\nabla}\bar{f}|(p)&\le \frac{n-1}{r_0}.\label{computation of laplacian u-5}
\end{align}
Note that $\bar{u}(\hat{x_1})<0$ and $\psi^{\prime}\le0$. Combining (\ref{computation of laplacian u-3}), (\ref{computation of laplacian u-4}), (\ref{computation of laplacian u-5}) with Lemma\ref{lem-bound of laplacian d}, we get

\begin{align}\label{inequality of laplacian u }
\bar{\Delta}\bar{u}(\hat{x}_1)\le\frac{|\bar{u}(\hat{x}_1)|}{\psi}\Big( \frac{(\psi^{\prime})^2}{\psi}\frac{2}{(Ar_0)^2}-\frac{\psi^{\prime}}{Ar_0}\big(\frac{\lambda}{2}+\frac{8(n-1)}{3r_0}\big)+\frac{|\psi^{\prime\prime}|}{Ar_0^2}+\lambda-\frac{2}{n}|\bar{u}(\hat{x}_1)|\Big).
\end{align}

 By (\ref{bounds of psi}),(\ref{bar u at minimum}) and (\ref{inequality of laplacian u }), we get 
\begin{align}
    \bar{u}(\hat{x}_1)\ge -\frac{C+\lambda}{Ar^2_0}-\frac{n\lambda}{2}.
\end{align}
If $R(x)\le 0$ and $x\in B(p,\frac{Ar_0}{2})$, then
\begin{align}
R(x)\ge \frac{u(x)}{\psi(\frac{1}{2})}\ge \frac{\bar{u}(\hat{x}_1)}{\psi(\frac{1}{2})}\ge -\frac{C+\lambda}{Ar_0^2\psi(\frac{1}{2})}-\frac{n\lambda}{2\psi(\frac{1}{2})}.\notag
\end{align}
Hence,
\begin{align}\label{lower bound estimate of R-1}
    R(x)\ge \min\{0,-\frac{C+\lambda}{Ar_0^2\psi(\frac{1}{2})}-\frac{n\lambda}{2\psi(\frac{1}{2})}\},~\forall~x\in B(p,\frac{Ar_0}{2}).
\end{align}

\textbf{ii)} Suppose $\hat{x}_1$ is in the cut locus of $p$ on $(N,\bar{g})$. Let $\bar{u}(x)=\psi(\frac{\bar{d}(x,p)}{Ar_0})\bar{R}(x)$ for $x\in \hat{B}(\hat{x_1},\delta)$. By the same argument as in case \textbf{i)}, it remains true that 
  \begin{align}
    \bar{u}(x_1)=\psi(\frac{\bar{d}(\hat{x}_1,p)}{Ar_0})\bar{R}(\hat{x_1})=\psi(\frac{d(x_1,p)}{Ar_0})R(x_1)=u(x_1)<0.
\end{align}
However, $\bar{u}(x)$ may not be smooth around $\hat{x}_1$. We now contruct a new function on $\hat{B}(\hat{x_1},\delta)$ which is smooth around $\hat{x}_1$ and attains its minimum at $\hat{x}_1$.

Since $\hat{x}_1$ is in the cut locus of $p$ on $(N,\bar{g})$ and $\bar{\gamma}(s)$ is minimal geodesic on $N$ from $p$ to $x_1$. For any $\epsilon>0$, $\bar{\gamma}|_{[\epsilon,d(p,x_1)]}$ is a minimal geodesic from $\bar{\gamma}(\epsilon)$ to $\hat{x}_1$ and $\hat{x}_1$  is not in the cut locus of $\bar{\gamma}(\epsilon)$. Now we fix $\epsilon$ such that $\epsilon\le \varepsilon$. Let $p^{\prime}=\bar{\gamma}(\epsilon)$ and $v(x)=\psi(\frac{\bar{d}(x,p^{\prime})+\epsilon}{Ar_0})\bar{R}(x)$ for $x\in \hat{B}(\hat{x_1},\delta)$. Note that 
\begin{align}
    \bar{d}(x,p^{\prime})+\epsilon= \bar{d}(x,p^{\prime})+\bar{d}(p,p^{\prime})\ge\bar{d}(x,p)\\
    \bar{d}(\hat{x}_1,p^{\prime})+\epsilon= \bar{d}(\hat{x}_1,p^{\prime})+\bar{d}(p,p^{\prime})=\bar{d}(\hat{x}_1,p)
\end{align}
When $\delta$ is small, by (\ref{scalar curvature negative}), we can assume that 
\begin{align}
    \bar{R}(x)<0,~\forall~x\in \bar{B}(\hat{x_1},\delta).
\end{align}
Therefore, 
\begin{align}
v(x)=\psi(\frac{\bar{d}(x,p^{\prime})+\epsilon}{Ar_0})\bar{R}(x)\ge \psi(\frac{\bar{d}(x,p)}{Ar_0})\bar{R}(x)=\bar{u}(x),~ \forall ~x\in \hat{B}(\hat{x_1},\delta). \\
  v(\hat{x}_1)=\psi(\frac{\bar{d}(\hat{x}_1,p^{\prime})+\epsilon}{Ar_0})\bar{R}(\hat{x}_1)=\psi(\frac{\bar{d}(\hat{x}_1,p)}{Ar_0})\bar{R}(\hat{x_1})=\bar{u}(\hat{x}_1)<0.
\end{align}
Hence, $v(x)$ attains its minimum at $\hat{x}_1$ and $v(x)$ is smooth around $\hat{x}_1$.
 By (\ref{computation of laplacian u-4}) and (\ref{computation of laplacian u-5}), when $\epsilon$ is small, we have 
\begin{align}
\bar{\mbox{Ric} }(x)&\le \frac{n-1}{r_0^2},~\forall~x\in \bar{B}(p^{\prime},r_0-\epsilon).\label{computation of laplacian v-1}\\
|\bar{\nabla}\bar{f}|(p^{\prime})&\le \frac{2(n-1)}{r_0}.\label{computation of laplacian v-2}
\end{align}
By (\ref{computation of laplacian v-1}) and (\ref{computation of laplacian v-2}), we can follow the argument in case \textbf{i)} to compute $\bar{\Delta} v$ and show that 
\begin{align}\label{lower bound estimate of R-2}
    R(x)\ge\min\{0,-\frac{C^{\prime}+\lambda}{A(r_0-\epsilon)^2\psi(\frac{1}{2})}-\frac{n\lambda}{2\psi(\frac{1}{2})}\},~\forall~x\in B(p,\frac{Ar_0}{2}).
\end{align}

Finally, by (\ref{lower bound estimate of R-1-1}),(\ref{lower bound estimate of R-1-2}), (\ref{lower bound estimate of R-1}) and (\ref{lower bound estimate of R-2}), we have 
\begin{align}
    R(x)\ge\min\{0,-\frac{C+C^{\prime}+2|\lambda|}{A(r_0-\epsilon)^2\psi(\frac{1}{2})}-\frac{n\lambda}{2\psi(\frac{1}{2})},-\frac{n\lambda}{2\psi(\frac{1}{2})}\},~\forall~x\in B(p,\frac{Ar_0}{2}) .\notag
\end{align}
By taking $A\to\infty$, 
\begin{align}
    R(x)\ge\min\{0,-\frac{n\lambda}{2\psi(\frac{1}{2})}\},~\forall~x\in M .\notag
\end{align}
For $\lambda\le0$, $R(x)\ge 0$ on $M$. For $\lambda>0$, 
\begin{align}
    R(x)\ge -\frac{n\lambda}{2\psi(\frac{1}{2})},~\forall~x\in M .\notag
\end{align}
We complete the proof.
\end{proof}

In the end of this section, we prove Theorem \ref{theo-automorphism}.

  \begin{proof}[Proof of Theorem \ref{theo-automorphism}]
  By Theorem \ref{theo-R has lower bound}
 and Lemma \ref{lem-completeness when R has lower bound}, $\phi_t(p)$ exists for all $p\in M$ and $t\in \mathbb{R}$.      We first show that $\phi_t$ is a smooth map from $\mathcal{M}$ to $\mathcal{M}$ for any $t\in\mathbb{R}$. We may suppose $t>0$. By Definition \ref{def-phi}, for any $p\in M$, suppose $\phi_t(p)=\phi_{t_m}(\phi_{t_{m-1}}(\cdots\phi_{t_1}(p)\cdots))$ and the following hold:
  
  (1) $t=t_1+t_2+\cdots+t_m$ and $t\cdot t_{i}\ge 0$ for all $1\le i\le m$;

  (2) For $1\le i\le m$, there exists a chart $(U_i,\hat{U}_i,\pi_i,\Gamma_i)$ such that  $p_1=p$, $p_{i+1}=\phi_{t_{i}}(p_i)$  and $\phi_{s}(p_i)\in U_i$  when $s\in[0,t_i]$, where $\phi_{s}(p_i)\in U_i$ is defined in Definition \ref{def-phi-local} .

  In each local chart $(U_i,\hat{U}_i,\pi_i,\Gamma_i)$, actually, $\phi_t=(\hat{\phi}_t,Id_{\Gamma_i})$, where $\hat{\phi}_t$ is generated by $\hat{\nabla}\hat{f}$ and $\hat{f}=f\circ \pi_{i}$.
By Lemma \ref{lem-commutative lemma} and Lemma \ref{lem-commutative lemma-singular}, we have 
\begin{align}
\hat{\phi}_{t}(\hat{p})\cdot h=\hat{\phi}_{t}(\hat{p}\cdot h),~\forall~t\in (a,b),~h\in\Gamma_i.
\end{align}
Obviously, $\hat{\phi}_t$ is smooth. Hence, $\phi_t$ is smooth in each chart $(U_i,\hat{U}_i,\pi_i,\Gamma_i)$. Therefore, $\phi_t(p)=\phi_{t_m}(\phi_{t_{m-1}}(\cdots\phi_{t_1}(p)\cdots))$ is a smooth orbifold map from $\mathcal{M}$ to $\mathcal{M}$. 

Suppose $\phi_t(p)=x$. Note that $\phi_{-t}(x)=\phi_{-t_1}(\phi_{-t_{2}}(\cdots\phi_{-t_m}(x)\cdots))$ and $\phi_{t_i},\phi_{-t_i}$ are both defined in chart   $(U_i,\hat{U}_i,\pi_i,\Gamma_i)$. It is easy to see that $\phi_t(\phi_{-t})=Id$ and $\phi_{-t}(\phi_{t})=Id$. Hence, $\phi_t$ is an orbifold automorphism of $\mathcal{M}$.
      
  \end{proof}

\section{Rigidity of positively curved steady GRS on orbifolds}\label{section-rigidity of soliton}

The existence of  equilibrium points (i.e. the points where $|\nabla f|=0$ holds) may affect the topology of the soliton. We note that gradient Ricci solitons with compact singularity always admit an equilibrium point.
\begin{lem}\label{lem-existence of stable point}
Suppose $(\mathcal{M},g,f)$  is a gradient Ricci soliton on an $n$-dimens-ional Riemannian orbifold. Let $\Sigma\subset M_{\mbox{sing}}$ be a compact connected component of $M_{\mbox{sing}}$. If $\Sigma\subset M_{\mbox{sing}}$ is non-empty, then there exists one point $p\in \Sigma$ such that $|\nabla f|(p)=0$.
\end{lem}
\begin{proof}
Since $f$ is continuous on $M$, it is also continous on the compact subset $\Sigma$. Then, there exists a point $p\in M$ such that $f(p)=\min_{x\in \Sigma}f(x)$.  By Lemma\ref{lem-singular point stay in singular set}, $\phi_t(p)\in \Sigma$, $\forall~t\in(-\infty,+\infty)$. Note that 
\begin{align}\label{differential of f}
    \frac{d f(\phi_t(p))}{d t}=|\nabla f|^2(\phi_t(p))\ge 0.
\end{align}
Hence,
\begin{align}
    f(\phi_t(p))\le f(p)=\min_{x\in \Sigma}f(x),~\forall~t\in (-\infty,0].\notag
\end{align}
It follows that 
\begin{align}
    f(\phi_t(p))=f(p),~\forall~t\in (-\infty,0].\notag
\end{align}
Therefore, $|\nabla f|(p)=0$ by (\ref{differential of f}).

\end{proof}

The following lemma shows that the equilibrium point is unique on steady Ricci solitons with positive Ricci curvature.

\begin{lem}\label{lem-uniqueness of stable point}
Suppose $(\mathcal{M},g,f)$  is a steady gradient Ricci soliton on an $n$-dimensional orbifold with positive Ricci curvature.  Then, $(\mathcal{M},g,f)$ admits at most one point $p$ such that $|\nabla f|(p)=0$.  
\end{lem}
\begin{proof}
    Suppose there exist two points $p,q\in M$ such that $|\nabla f|(p)=|\nabla f|(q)=0$.  There is a minimal geodesic $\gamma(s):[0,d(p,q)]\to M$  from $p$ to $q$ and $\gamma(s)\in M$ for $s\in[0,d(p,q))$. Since $\phi_t$ is an isomorphism from $(\mathcal{M},g(t))$ to $(\mathcal{M},g)$. Note that $d_{g(t)}(p,q)=d(p,q)$ for all $t\in(-\infty,+\infty)$. Then, the length of $\gamma(s)$ w.r.t $g(t)$ satisfy
    \begin{align}
        \frac{d L(\gamma;g(t))}{dt}=\int_{0}^{d(p,q)}|\gamma^{\prime}(s)|^2_{g(t)}ds=-2\int_{0}^{d(p,q)}\mbox{Ric}_{g(t)}(\gamma^{\prime}(s),\gamma^{\prime}(s))ds<0.
    \end{align}
    Hence,
    \begin{align}
        d_{g(t)}(p,q)<L(\gamma;g(0))=d(p,q)
    \end{align}
    It contradicts the fact that $d_{g(t)}(p,q)=d(p,q)$ for all $t\in(-\infty,+\infty)$.

\end{proof}

Now, we  are ready to prove a structure theorem for steady Ricci solitons on orbifolds with positive Ricci curvature.

\begin{theo}\label{thm-structure of obifold solition}
Suppose $(\mathcal{M},g,f)$  is a steady gradient Ricci soliton on an $n$-dimensional Riemannian orbifold with positive Ricci curvature and essentially compact singular set\footnote{An orbifold $\mathcal{M}$ has essentially compact singular set if every connect component of 
 the singular locus $M_{\mbox{sing}}$ is compact. }. If $(\mathcal{M},g,f)$ admits a point $p$ with $|\nabla f|(p)=0$, then $M=\hat{M}/\Gamma$ for some finite group $\Gamma\subset O(n)$, where $(\hat{M},\hat{g},\hat{f})$ is a steady gradient Ricci soliton on a smooth manifold.
\end{theo}
\begin{proof}
    By Lemma\ref{lem-existence of stable point} and Lemma\ref{lem-uniqueness of stable point}, it is easy to see  that $M_{\mbox{sing}}$ has only one connected component. Let $p_1$ and $p_2$ be the points such that $f(p_1)=\max_{x\in M_{\mbox{sing}}}f(x)$ and $f(p_2)=\min_{x\in M_{\mbox{sing}}}f(x)$. By the proof of Lemma\ref{lem-existence of stable point}, $|\nabla f|(p_1)=|\nabla f|(p_2)=0$. By Lemma\ref{lem-uniqueness of stable point}, $p_1=p_2$. So $f(x)$ is a constant on $M_{\mbox{sing}}$. For any $x\in M_{\mbox{sing}}$, $\phi_t(x)\in M_{\mbox{sing}}$ by Lemma\ref{lem-singular point stay in singular set}. So, $f(\phi_t(x))=f(x)$ for $t\in(-\infty,+\infty)$. Now, we get 
    \begin{align}
         |\nabla f|^2(x)= \frac{d f(\phi_t(x))}{d t}|_{t=0}= 0.
    \end{align}
   By Lemma\ref{lem-uniqueness of stable point}, $x=p$. Hence, we get $M_{\mbox{sing}}=\{p\}$.

  Let $(U,\hat{U},\pi,\Gamma)$ be  a local model  around $p\in U$ along with $\hat{p}\in \hat{U}$ project to $p$ such that $\Gamma=\Gamma_p$. Let $\pi:\hat{U}\to U$ be the projection and $\hat{f}=f\circ\pi$. Note that $\hat{\nabla} \hat{f}(\hat{p})=0$ and $\hat{p}$ is the only equilibrium point in $\hat{U}$. Since the Ricci curvature is positive, we get that $\{\hat{f}=a\}$ is diffeomorphic to $\mathbb{S}^{n-1}$ for $a> \hat{f}(\hat{p})$ by Lemma 2.3 in \cite{DZ7}. Note that $\{\hat{f}=a\}/\Gamma$ is diffeomorphic to $\{f=a\}$, for any $a> \hat{f}(\hat{p})$. Since $|\nabla f|(p)=0$ and Ricci curvature is positive, following Proposition 2.3 in \cite{CaCh} we get that there exist constants $C_1$ and $C_2$ such that 
  \begin{align}
      C_1r(x)\le  f(x)\le C_2 r(x),~\forall~r(x)\ge r_0.
  \end{align}
  Therefore, $f(M)=[f(p),+\infty)$.
  
  Fix an $a> f(p)$. Note that $|\nabla f|(x)\neq0$ when $f(x)\ge f(p)$.  Since $p$ is the only singular point, $\{f(x)>a\}$ is a smooth manifold. On $M\setminus\{p\}$, we can define $\Phi_t$ which is generated by $\frac{\nabla f}{|\nabla  f|^2}$. Now, we define $\iota:\{f=a\}\times (a,+\infty)\rightarrow\{f>a\}$ such that $\iota(x,s)=\Phi_{s-a}(x)$, for all $(x,s)\in \{f=a\}\times (a,+\infty)$. It's easy to check that $\iota^{-1}(y)=\Phi_{a-f(y)}(y)$. Hence, $\iota$ is a diffeomorphism from $\{f=a\}\times (a,+\infty)$ to $\{f>a\}$. Later on, we will stick $\{\hat{f}<b\}$ and the universal covering of $\{f=a\}\times (a,+\infty)$  to make the universal covering of $\mathcal{M}$.
   
  Now we choose $b>a$ such that $\{ \hat{f}\le b\}\subset \hat{U}$. Note that $|\hat{\nabla} \hat{f}|(x)\neq0$ when $\hat{f}(x)\ge \hat{f}(\hat{p})$. Since $\Gamma=\Gamma_p$, $\hat{p}$ is the only point in $\hat{U}$ such that $|\hat{\nabla} \hat{f}|(\hat{p})=0$. Let $\hat{\Phi}_t$ be generated by $\frac{\hat{\nabla} \hat{f}}{|\hat{\nabla}  \hat{f}|^2}$ on $\hat{U}\setminus\{\hat{p}\}$. Similarly, we can also define $\Psi:\{\hat{f}=a\}\times (a,b)\rightarrow\{a<\hat{f}<b\}$ such that $\Psi(\hat{x},s)=\hat{\Phi}_{s-a}(\hat{x})$, for all $(\hat{x},s)\in \{\hat{f}=a\}\times (a,b)$. Since $\Psi^{-1}(\hat{y})=\hat{\Phi}_{a-\hat{f}(\hat{y})}(\hat{y})$, $\Psi$ is a diffeomorphism from $\{\hat{f}=a\}\times (a,b)$ to $\{a<\hat{f}<b\}$. 
       
       Define $\pi^{\prime}(\hat{x},s)=(\pi(\hat{x}),s)$. We claim that $\iota\circ\pi^{\prime}=\pi\circ\Psi$, i.e. the following commutative diagram holds.

        \begin{displaymath}
  \xymatrix{
\{\hat{f}=a\}\times(a,b)\ar[r]^{\quad\Psi}\ar[d]_{\pi^{\prime}}&
  \{a<\hat{f}<b\} \ar[d]^{\pi}\\
 \{f=a\}\times(a,b)\ar[r]_{\quad\iota}&\{a<f<b\}
  }
  \end{displaymath}
  Let $\hat{\phi}_t$ be generated by $\hat{\nabla}\hat{f}$  on $\hat{U}$ and $\phi_t$ be generated by $\nabla f$ on $M\setminus\{p\}$\footnote{$M\setminus\{p\}$ can be regard as a manifold.}. Note that the integral curves of $\hat{\nabla}\hat{f}$ and  $\frac{\hat{\nabla} \hat{f}}{|\hat{\nabla}  \hat{f}|^2}$ have the same image. Therefore, for any $(\hat{x},s)\in \{\hat{f}=a\}\times(a,b)$, there exists a unique constant $t_s$ such that 
  \begin{align}\label{eq:same image}
      \hat{\Phi}_{s-a}(\hat{x})=\hat{\phi}_{t_s}(\hat{x}).
  \end{align}
  By Lemma \ref{lem-regualar point stay regular} and (\ref{eq:same image}),  
  \begin{align}\label{eq:a}
      \pi\circ\Psi(\hat{x},s)=\pi\circ\hat{\Phi}_{s-a}(\hat{x})=\pi\circ\hat{\phi}_{t_s}(\hat{x})=\phi_{t_s}(x),
  \end{align}
  where $x=\pi(\hat{x})$.

On the other hand, the integral curves of $\nabla f$ and  $\frac{\nabla f}{|\nabla  f|^2}$ have the same image. So, there exists a unique constant $\tau_s$ such that 
\begin{align}\label{eq:d}
    \Phi_{s-a}(x)=\phi_{\tau_s}(x).
\end{align}
It follows that
\begin{align}\label{eq:b}
    \iota\circ\pi^{\prime}(\hat{x},s)=\iota(x,s)=\Phi_{s-a}(x)=\phi_{\tau_s}(x).
\end{align}
By integrating $\hat{f}(\hat{\Phi}_{t}(\hat{x}))$ and  $f(\Phi_{t}(x))$ along the integral curves of $\frac{\hat{\nabla} \hat{f}}{|\hat{\nabla}  \hat{f}|^2}$ and $\frac{\nabla f}{|\nabla  f|^2}$, we have 
\begin{align}\label{eq:c}
    \hat{f}(\hat{\Phi}_{s-a}(\hat{x}))=s=f(\Phi_{s-a}(x)).
\end{align}
By (\ref{eq:a}), (\ref{eq:d}) and (\ref{eq:c}),
\begin{align}
    f(\phi_{t_s}(x))=\hat{f}(\hat{\Phi}_{s-a}(\hat{x}))=f(\Phi_{s-a}(x))=f(\phi_{\tau_s}(x)).\notag
\end{align}
Note that $f(\phi_t(p))$ is a monotone function with respect to $t$. Hence, 
\begin{align}\label{eq:e}
t_s=\tau_s.
\end{align}
By (\ref{eq:a}), (\ref{eq:b}) and (\ref{eq:e}), we get $\iota\circ\pi^{\prime}=\pi\circ\Psi$. The claim is complete.
  
   By giving these mainfolds with the induced metrics, we get the following commutative diagram.
  \begin{displaymath}
  \xymatrix{
(\{\hat{f}=a\}\times(a,b),(\iota\circ \pi^{\prime})^{\ast}g)\ar[r]^{\quad\quad\Psi}\ar[d]_{\pi^{\prime}}&
  (\{a<\hat{f}<b\},\hat{g}) \ar[d]^{\pi}\\
 (\{f=a\}\times(a,b),\iota^{\ast}g)\ar[r]_{\quad\iota}&(\{a<f<b\},g)
  }
  \end{displaymath}
Note that $\Psi$ is an isomorphism. The isomorphism $\iota$ and the projection $\pi^{\prime}$ actually have a larger domain of definition  as the following.

\begin{displaymath}
  \xymatrix{
  (\{\hat{f}=a\}\times(a,+\infty),(\iota\circ \pi^{\prime})^{\ast}g)\ar[dr]^{\iota\circ \pi^{\prime}}\ar[d]_{\pi^{\prime}}\\
 (\{f=a\}\times(a,+\infty),\iota^{\ast}g)\ar[r]_{\qquad\iota}&(\{f>a\},g)
  }
  \end{displaymath}

 Let $\hat{V}=\{\hat{f}<b\}$. Now,  we stick $\hat{V}$ and $\{\hat{f}=a\}\times (a,+\infty)$ via the isomorphism $\Psi$  to make the universal covering of $\mathcal{M}$. we can define $\tilde{M}=\hat{V}\sqcup_{\Psi}(\{\hat{f}=a\}\times(a,+\infty))$. For $x\in \hat{V}$, let $\tilde{g}=\hat{g}$. For $x\in \{\hat{f}=a\}\times(a,+\infty)$, let $\tilde{g}=(\iota\circ \pi^{\prime})^{\ast}g$. We  define $\tilde{f}$ as following.
 $$\tilde{f}=\begin{cases}
 \hat{f}, &\text{on}\quad \hat{V},\\
 f\circ\iota\circ\pi^{\prime},&\text{on} \quad\{\hat{f}=a\}\times(a,+\infty).
 \end{cases}
 $$
 Then, it is easy to check that $(\tilde{M},\tilde{g},\tilde{f})$ is a steady gradient Ricci soltion and $M=\tilde{M}/\Gamma$.

\end{proof}

In the end, we prove Theorem\ref{main-1} and Theorem\ref{main-2}. We first recall a lemma on the potiential $f$.

\begin{lem}\label{lem-growth of f}
Suppose $(\mathcal{M},g,f)$  is a non-Ricci-flat steady gradient Ricci soliton on an $n$-dimensional Riemannian orbifold with compact singular set. Suppose the scalar curvaure of $(M,g)$ has uniform decay and the scalar curvature satisfies
\begin{align}\label{eq: dRdt geq 0 on M minus K}
\left.\frac{\rm d}{{\rm d} t}\right|_{t=0}R(\phi_t(p))\ge 0 ~\text{ for all }p\in M\setminus K,
\end{align}
where $K$ is a compact subset of $M$. Then, $f$ satifies
\begin{align}\label{linear of f}
C_1r(x)\le f(x)\le C_2 r(x),~\forall~r(x)\ge r_0,
\end{align}
where $C_1$, $C_2$ and $r_0$ are constants. $r(x)$ is the distance function from a fixed point $x_0$.
\end{lem}
\begin{proof}
 Outside a large compact set, we may treat the orbifold as a manifold under our assumption.  Applying the argument in the proof of Theorem 2.1 in \cite{CDM22}, we may conclude (\ref{linear of f}).
\end{proof}

Now we prove Theorem\ref{main-1}.

\begin{proof}[Proof of Theorem\ref{main-1}]
By Lemma \ref{lem-existence of stable point}, there exists a point $p_0$ such that $|\nabla f|(p_0)=0$.
By Theorem\ref{thm-structure of obifold solition}, we may assume that $M=\hat{M}/\Gamma$ for some finite group $\Gamma\subset O(n)$ and $(\hat{M},\hat{g},\hat{f})$ is a smooth steady gradient Ricci soliton. Therefore, $(\hat{M},\hat{g},\hat{f})$ is also $\kappa-$noncollapsed with positive curvature operator. Note the scalar curvature satisfies
\begin{align}
R(x)\le \frac{C_1}{r(x)},~\forall~r(x)\ge r_0.
\end{align}
where $r(x)$is the distance function from a fixed point $x_0$ and $C_1$, $r_0$ are constants. Note that the positivity of Ricci curvature implies condition (\ref{eq: dRdt geq 0 on M minus K}).  By Lemma\ref{lem-growth of f}, we have
\begin{align}
R(x)\le \frac{C_2}{f(x)},~\forall~f(x)\ge r_1,
\end{align}
for some constants $r_1$ and $C_2$. Suppose $\pi:\hat{M}\to M$ is the covering map. Note that 
\begin{align}
\hat{R}(\hat{x})=R(x),~\hat{f}(\hat{x})=f(x),
\end{align}
if $\pi(\hat{x})=x$. It follows that
\begin{align}\label{linear decay-1}
\hat{R}(x)=R(\pi(x))\le \frac{C_2}{f(\pi(x))}=\frac{C_2}{\hat{f}(x)},~\forall~x\in \hat{M}.
\end{align}
Suppose $\pi(\hat{p}_{0}) =p_0$.  Then, $|\hat{\nabla}\hat{f}|(\hat{p}_{0})=|\nabla f|(p_{0})=0$.
Then, $(\hat{M},\hat{g},\hat{f})$ has positive Ricci curvature and an equilbrium point. Hence, $\hat{f}$ has linear growth by \cite{CaCh}. By \ref{linear decay-1}, we get
\begin{align}
\hat{R}(x)\le \frac{C_3}{\hat{r}(x)},~\forall~\hat{r}(x)\ge r_2,
\end{align}
where $\hat{r}(x)$is the distance function from a fixed point $\hat{x}_0$ and $C_3$, $r_2$ are positive constants. By Theorem 1.2 in \cite{DZ5}, we get that $(\hat{M},\hat{g},\hat{f})$ is isometric to the Bryant soliton. 

\end{proof}

We need the following lemma on the curvature decay of scalar curvature to prove Theorem\ref{main-2}.

\begin{lem}\label{lem-decay}
Let $(\mathcal{M},g,f)$ be a non-Ricci-flat steady gradient Ricci soliton on an $n$-dimensional Riemannian orbifolds with compact singular set. Suppose that the scalar curvature decays uniformly and that
\begin{align}\label{positive lower bound of derivative}
\frac{1}{R^2(p)}\cdot\left.\frac{\rm d}{{\rm d} t}\right|_{t=0 }R(\phi_t(p))\ge \epsilon>0\,\text{ for all }\,p\in M\setminus K,
\end{align}
where $K$ is a compact subset of $M$,
$\phi_t$ is generated by $-\nabla f$,
and $\epsilon$ is independent of $p,t$. Then there exist constants $r_0$ and $c$ such that
\begin{align}\label{linear decay}
 R(x)\le \frac{c}{r(x)},~\forall~r(x)\ge r_0.
\end{align}
 where $r(x)$ is the distance function from a fixed point $x_0\in M$.
\end{lem}
\begin{proof}
 Outside a large compact set, we may treat the orbifold as a manifold under our assumption.  Applying the argument in the proof of Theorem 3.1 in \cite{CDM22}, we may conclude (\ref{linear of f}).
\end{proof}

Now we prove Theorem\ref{main-2}.

\begin{proof}[Proof of Theorem\ref{main-2}]

 We first show that $(\mathcal{M},g,f)$ has linear curvature decay and $f$ has linear growth. Note that 
\begin{align}\label{condition for linear decay}
\frac{1}{R^2(p)}\cdot\left.\frac{\rm d}{{\rm d} t}\right|_{t=0 }R(\phi_t(p))=\frac{\Delta R(p)+2|{\Ric}|^2(p)}{R^2(p)}.
\end{align}
 
By Lemma \ref{lem-decay}, Lemma \ref{lem-growth of f} and (\ref{condition for linear decay}). It suffices to show  
\begin{align}\label{inequality-a}
\frac{\Delta R(p)+2|{\Ric}|^2(p)}{R^2(p)}\ge \epsilon>0\,\text{ for all }\,p\in M\setminus K,
\end{align}
for some compact set $K$.

We prove (\ref{inequality-a}) by contradiction. 
By our assumption, for any $p_i\to\infty$, $(M,g_i(t),p_i)$ converge to $(\mathbb{S}^{n-1}/\Gamma\times\mathbb{R},g_{\infty}(t),p_{\infty})$, where $g_{i}(t)=R(p_i)g(R^{-1}(p_i)t)$ and $g_{\infty}(t)=g_{\mathbb{S}^{n-1}/\Gamma}(t)+ds^2$. By the convergence, we have
\begin{align}\label{limit-2}
\frac{\Delta_{\infty} R_{\infty}(p_{\infty})+2|{\Ric}_{\infty}|^2(p_{\infty})}{R_{\infty}^2(p_{\infty})}=\lim_{i\to\infty}\frac{\Delta R(p_i)+2|{\Ric}|^2(p_i)}{R^2(p_i)}.
\end{align}
Suppose (\ref{inequality-a}) doesn't hold. Then, there exists a sequence of points $p_i$ tending to infinity such that
\begin{align}
\lim_{i\to\infty}\frac{\Delta R(p_i)+2|{\Ric}|^2(p_i)}{R^2(p_i)}\le0.\notag
\end{align}
Therefore,
\begin{align}\label{limit-1}
\frac{\Delta_{\infty} R_{\infty}(p_{\infty})+2|{\Ric}_{\infty}|^2(p_{\infty})}{R_{\infty}^2(p_{\infty})}\le0.
\end{align}
On the other hand, we can compute directly to show that 
\begin{align}
\frac{\Delta_{\infty} R_{\infty}(p_{\infty})+2|{\Ric}_{\infty}|^2(p_{\infty})}{R_{\infty}^2(p_{\infty})}=\frac{2}{n-1}.\notag
\end{align}
It contradicts (\ref{limit-1}).

Hence, we have proved (\ref{inequality-a}). By lemma \ref{lem-decay}, we have 
\begin{align}\label{linear decay inequality}
 R(x)\le \frac{c}{r(x)},~\forall~r(x)\ge r_0.
\end{align}

By Lemma \ref{lem-existence of stable point}, there exists a point $p_0$ such that $|\nabla f|(p_0)=0$.
By Theorem\ref{thm-structure of obifold solition}, we may assume that $M=\hat{M}/\Gamma$ for some finite group $\Gamma\subset O(n)$ and $(\hat{M},\hat{g},\hat{f})$ is a smooth steady gradient Ricci soliton with nonnegative sectional curvature and positive Ricci curvature. Note that we already get the linear decay (\ref{linear decay inequality}) now. We can apply the argument in the proof of Theorem \ref{main-1} to show
\begin{align}
\hat{R}(x)\le \frac{c_1}{\hat{r}(x)},~\forall~\hat{r}(x)\ge r_1,
\end{align}
where $\hat{r}(x)$ is the distance function from a fixed point $\hat{x}_0$ and $c_1$, $r_1$ are positive constants. 

We also claim $(\hat{M},\hat{g},\hat{f})$ is $\kappa$-noncollapsed. It suffices to show that $(M,g,f)$ is $\kappa$-noncollapsed. We prove by contradiction. If $(M,g,f)$ is not $\kappa$-noncollapsed for any $\kappa>0$, then there exist $x_i\in M$ and $r_i>0$ such that 
\begin{align}
R(x_i)\le \frac{1}{r^2_i},~\forall~i\in\mathbb{N},
\end{align}
and
\begin{align}\label{limit=0}
\lim_{i\to\infty}\frac{V(x_i,r_i;g)}{r_i^n}=0,
\end{align}
where $V(x_i,r_i;g)$ is the volume of geodesic ball $B(x_i,r_i;g)$ with respect to the meric $g$.

If there exists a compact set $K_0$ sucht that $x_i\in K_0$ for all $i$, then there exists a positive constant $c>0$ such that $R(x_i)>c$ for all $i$. Therefore, $r_i<\frac{1}{\sqrt{c}}$ for all $i\in \mathbb{N}$. Let $\pi:\hat{M}\to M$ and $\hat{K}_0=\pi^{-1}(K_0)$. Since $(\hat{M},\hat{g})$ is a smooth manifold with positive Ricci curvature, by volume comparison theorem, there exists a positive constant $c_1$ such that
\begin{align}
\frac{V(x,r;\hat{g})}{r^n}\ge c_1,~\forall~r\le \frac{1}{\sqrt{c}},~\forall x\in \hat{K}_0.\notag
\end{align}
Note that $\pi$ is a finite covering map. Then, we have 
\begin{align}
V(\hat{x}_i,r_i;\hat{g})\ge V(x_i.r_i;g)\ge \frac{1}{|\Gamma|}V(\hat{x}_i,r_i;\hat{g}),\notag
\end{align}
where $\hat{x}_i$ satisfies $\pi(\hat{x}_i)=x_i$.
It follows that 
\begin{align}
\frac{V(x_i,r_i;g)}{r_i^n}\ge \frac{1}{|\Gamma|}\cdot\frac{V(\hat{x}_i,r_i;\hat{g})}{r_i^n}\ge \frac{c_1}{|\Gamma|}\notag
\end{align}
It contradicts (\ref{limit=0}). Then, $\{x_i\}_{i\in\mathbb{N}}$ is unbounded.

We may suppose $x_i\to \infty$. By our assumption, $(M,R(x_i)g,x_i)$ converge to $(\mathbb{S}^{n-1}/\Gamma\times\mathbb{R},g_{\infty},x_{\infty})$, where $g_{\infty}=g_{\mathbb{S}^{n-1}/\Gamma}+ds^2$ and $g_{\mathbb{S}^{n-1}/\Gamma}$ has constant sectional curvature. By the convergence, we have 
\begin{align}
\lim_{i\to\infty}V(x_i,1;R(x_i)g)=V(x_{\infty},1;g_{\infty})\triangleq c_2>0.\notag
\end{align}
Then, there exists $i_0\in\mathbb{N}$ such that 
\begin{align}
V(x_i,1;R(x_i)g)>\frac{c_2}{2},~\forall~i\ge i_0.\notag
\end{align}
 Since the limit of $(M,R(x_i)g,x_i)$ is a smooth manifold, $B(x_i,1;R(x_i)g)$ also lies in the regular part of $M$ when $i$ large enough. Note the Ricci curvature is positive on $M$. By volume comparison theorem, we have
\begin{align}
\frac{V(x_i,r;R(x_i)g)}{r^n}>\frac{c_2}{2},~\forall~i\ge i_0,~\forall~0<r\le 1.\notag
\end{align}
Recall $R(x_i)r_i^2\le 1$. Now, by taking $r=r_i\sqrt{R(x_i)}$, we have 
\begin{align}
\frac{V(x_i,r_i;g)}{r_i^n}=\frac{V(x_i,r;R(p_i)g)}{r^n}>\frac{c_2}{2}.\notag
\end{align}
It contradicts (\ref{limit=0}). So, $(\hat{M},\hat{g},\hat{f})$ is $\kappa$-noncollapsed for some $\kappa>0$.

Now, we see that $(\hat{M},\hat{g},\hat{f})$ is a $\kappa$-noncollpased steady gradient Ricci soliton with nonnegative sectional curvature, positive Ricci curvature and linear curvature decay. By Theorem 5.4 in \cite{DZ5}, for any sequence of points $p_i$ tending to infinity, we have $(\hat{M},\hat{g}_i(t)),p_i)$ converge to $(\Sigma\times\mathbb{R},\hat{g}_{\infty}(t),p_{\infty})$, where $\hat{g}_{i}(t)=\hat{R}(p_i)\hat{g}(\hat{R}^{-1}(p_i)t)$ and $\hat{g}_{\infty}(t)=\hat{g}_{\Sigma}(t)+ds^2$.  Moreover, $\Sigma$ is diffeomorphic to the level set of $\hat{f}$ and the  scalar curvature $\hat{R}_{\Sigma}(x,t)$ of $(\Sigma,\hat{g}_{\Sigma}(t))$ satisfies
 \begin{align}\label{decay of t}
 \hat{R}_{\Sigma}(x,t)\le\frac{C}{|t|}, ~\forall~x\in \Sigma,
 \end{align}
 where $C$ is a uniform constant. 

We claim that $\hat{g}_{\Sigma}(t)$ is a group of metrics with constant sectional curvature and $\Sigma$ is diffeomorpic to $\mathbb{S}^{n-1}$. By Lemma 2.3 in \cite{DZ7}, we see that $\Sigma$ is diffeomorphic to $\mathbb{S}^{n-1}$. Following the argument in Theorem 3.24 of \cite{BCDMZ}, we can see that $(M,g,f)$ has positive curvature operator outside a compact set $K_1$. Therefore, $(\hat{M},\hat{g},\hat{f})$ also has positive curvature operator outside a compact set $K_2$ as long as $\pi^{-1}(K_1)\subset K_2$. So, $(\Sigma,\hat{g}_{\Sigma}(t))$ has nonnegative\footnote{By a modification of the  computation in Theorem 3.24 of \cite{BCDMZ}, one can also show that $(\Sigma,\hat{g}_{\Sigma}(t))$  has positive curvature operator.} curvature operator and $\Sigma$ is diffeomorpic to $\mathbb{S}^{n-1}$. Note the  scalar curvature $\hat{R}_{\Sigma}(x,t)$ satisfies estimate (\ref{decay of t}). By a work of Ni \cite{Ni}, we see that $\hat{g}_{\Sigma}(t)$ is a group of metrics with constant sectional curvature on $\mathbb{S}^{n-1}$.

Finally, we conclude that $(\hat{M},\hat{g}_i(t)),p_i)$ converge to $(\mathbb{S}^{n-1}\times\mathbb{R},\hat{g}_{\infty}(t),p_{\infty})$, where $\hat{g}_{i}(t)=\hat{R}(p_i)\hat{g}(\hat{R}^{-1}(p_i)t)$ and $\hat{g}_{\infty}(t)=g_{\mathbb{S}^{n-1}}(t)+ds^2$. So, $(\hat{M},\hat{g},\hat{f})$ is asymptotically cylindrical and has positive sectional curvature. By Brendle's work \cite{Br2}, $(\hat{M},\hat{g},\hat{f})$ is isomeric to the Bryant soliton.

\end{proof}

\end{document}